\documentclass{amsart}

\usepackage{amsmath,amsthm,mathrsfs}
\usepackage{amssymb,amsfonts}
\usepackage{url,relsize,xcolor}
\usepackage{tikz,tikz-cd}
\usepackage{enumitem}
\usepackage{xr}
\usepackage{nicefrac,xspace,mathtools,marvosym,verbatim}

\usepackage[hidelinks]{hyperref}

\theoremstyle{plain}
\newtheorem*{theorem*}{Main Theorem}
\newtheorem{thrm}{Theorem}[section]
\newtheorem{theorem}{Theorem}[section]

\newtheorem{proposition}[theorem]{Proposition}
\newtheorem{lemma}[theorem]{Lemma}
\newtheorem{corollary}[theorem]{Corollary}

\numberwithin{equation}{section}
\newtheorem*{claim}{Claim}
\theoremstyle{definition}
\newtheorem{definition}[theorem]{Definition}

\newtheorem{question}[theorem]{Question}
\newtheorem{problem}[theorem]{Problem}

\theoremstyle{remark}

\DeclareMathOperator{\ThE}{Th_{\exists}}

\DeclareMathOperator{\Aut}{Aut}
\DeclareMathOperator{\TrAut}{TrAut}
\DeclareMathOperator{\dom}{dom}

\DeclareMathOperator{\Fin}{Fin}
\DeclareMathOperator{\ran}{ran}

\DeclareMathOperator{\diam}{diam}
\tikzset{
 symbol/.style={
 draw=none,
 every to/.append style={
 edge node={node [sloped, allow upside down, auto=false]{$#1$}}}
 }
}

\numberwithin{theorem}{section}

\newcommand{\bbN}{{\mathbb N}}
\newcommand{\bbR}{{\mathbb R}}

\newcommand{\bbZ}{{\mathbb Z}}

\newcommand{\cF}{{\mathcal F}}

\newcommand{\cP}{{\mathcal P}}
\newcommand{\cK}{{\mathcal K}}

\newcommand{\cB}{{\mathcal B}}

\newcommand{\cA}{{\mathcal A}}
\newcommand{\cQ}{{\mathcal Q}}

\newcommand{\<}{\langle}
\renewcommand{\>}{\rangle}

\newcommand{\fA}{\mathfrak A}

\newcommand{\fS}{\mathfrak S} % following de Cornulier's notation for almost (near) bijections
\newcommand{\fSs}{\mathfrak S^*} % de Cornulier has no notation for this 
\newcommand{\fSx}{\mathfrak S^\circ} % de Cornulier has no notation for this 
\newcommand{\rmR}{\mathrm {R}}
\newcommand{\rmZ}{\mathrm {Z}}
\newcommand{\rmS}{\mathrm {S}}

\newcommand{\rs}{\restriction}
\newcounter{my_enumerate_counter}
\newcommand{\pushcounter}{\setcounter{my_enumerate_counter}{\value{enumi}}}
\newcommand{\popcounter}{\setcounter{enumi}{\value{my_enumerate_counter}}}

\newcommand{\N}{{\mathbb N}}
\newcommand{\Z}{{\mathbb Z}}
\renewcommand{\P}{{\mathcal P}}

\newcommand{\s}{{\sigma}}
\newcommand{\w}{{\omega}}

\newcommand{\sub}{{\,\subseteq\,}}
\newcommand{\card}[1]{\left\lvert #1 \right\rvert}
\newcommand{\explicitSet}[1]{\left\lbrace #1 \right\rbrace}
\newcommand{\brackets}[1]{\left\langle #1 \right\rangle}
\newcommand{\set}[2]{\explicitSet{#1 \colon #2}}
\newcommand{\seq}[2]{\brackets{#1 \colon #2}}

\newcommand{\ch}{\ensuremath{\mathsf{CH}}\xspace}

\newcommand{\pnmf}{\nicefrac{\mathcal{P}(\N)}{\mathrm{Fin}}}
\newcommand{\pxmf}{\nicefrac{\mathcal{P}(X)}{\mathrm{Fin}}}

\newcommand{\parity}{\mathrm{par}}

\newcommand{\calL}{\mathcal L}

\newcommand{\cU}{\mathcal U}

\DeclareMathOperator{\ZFC}{{\mathsf {ZFC}}}

\DeclareMathOperator{\OCAT}{{\mathsf {OCA}_T}}

\DeclareMathOperator{\MA}{{\mathsf {MA}_{\aleph_1}}}

\DeclareMathOperator{\proj}{proj}
\DeclareMathOperator{\Span}{Span}
\DeclareMathOperator{\supp}{supp}

\newcommand{\cstar}{$\mathrm{C^*}$}

\newcommand{\cPNF}{\pnmf}
\DeclareMathOperator{\id}{id}

\newcommand{\cstu}{\mathrm{C}^*_u}

\newcommand{\roeq}{\mathrm{Q}^*_u}

\DeclareMathOperator{\FIX}{FIX} % Here too 

\DeclareMathOperator{\Index}{index}
\newcommand{\bt}{\mathbf t}

\newcommand{\toC}{\overset{{}_{\mathcal C}}{\to}}

\newcommand{\toa}{\overset{{}_{\alpha}}{\longrightarrow}}

\begin{document}
%%%%%%%%%%%%

%%%%%%%%%%%%
\title{Conjugating trivial automorphisms of $\mathcal P(\mathbb N) / \mathrm{fin}$}
%%%%%%%%%%%%
\author{Will Brian}
\address {
W. R. Brian\\
Department of Mathematics and Statistics\\
University of North Carolina at Charlotte\\
9201 University City Blvd.\\
Charlotte, NC 28223, 
(USA)}
\email{wbrian.math@gmail.com}
\urladdr{wrbrian.wordpress.com}
%%%%%%%%%%%%
\author{Ilijas Farah}
\address {
I. Farah\\
Department of Mathematics and Statistics\\
York University \\
4700 Keele Street \\
North York, Ontario, Canada, M3J 1P3 \\
and Matemati\v{c}ki Institut SANU \\
Kneza Mihaila 36 \\
11000 Beograd, P. P. 367, Serbia
}
\email{ifarah@yorku.ca}
\urladdr{ifarah.mathstats.yorku.ca}
%%%%%%%%%%%%
%%%%%%%%%%%%

%%%%%%%%%%%%
%%%%%%%%%%%%

%%%%%%%%%%%%
%\thanks{The first author is partially supported by NSF grant DMS-2154229. The second author is partially supported by NSERC}
%%%%%%%%%%%%

%%%%%%%%%%%%
\begin{abstract}
A trivial automorphism of the Boolean algebra $\mathcal P(\mathbb N) / \mathrm{Fin}$ is an automorphism induced by the action of some function $\mathbb N \rightarrow \mathbb N$. The forcing axiom $\mathsf{OCA}_{\mathrm{T}}$ implies all automorphisms are trivial, and therefore two trivial automorphisms are conjugate if and only if they have the same (modulo finite) orbit structure. We show that the Continuum Hypothesis implies that two trivial automorphisms are conjugate if and only if there are no first-order obstructions to their conjugacy and their indices have the same parity, if and only if the given trivial automorphisms are conugate in some forcing extension of the universe. 

To each automorphism $\alpha$ of $\pnmf$ we associate the first-order structure $\fA_\alpha=(\pnmf,\alpha)$  and compute the existential theories of these structures. 

These results are applied to resolve a question of Braga, Farah, and Vignati and prove that there are  coarse metric spaces $X$ and $Y$ such that the isomorphism between their uniform Roe coronas is independent of $\ZFC$.
\end{abstract}
%%%%%%%%%%%%
 
%%%%%%%%%%%%
\maketitle
%%%%%%%%%%%%

\section{Introduction}
%%%%%%%%%%%%

In this paper we study automorphisms of the Boolean algebra $\pnmf$ assuming the Continuum Hypothesis, \ch. 
Two automorphisms $\alpha$ and $\beta$ of $\pnmf$ are \emph{conjugate} if there is another automorphism $\gamma$ such that $\gamma \circ \alpha = \beta \circ \gamma$, or in other words, such that the following diagram commutes.

\begin{figure}[h]
\begin{tikzpicture}[scale=.8]

\node at (0,0) {$\pnmf$};
\node at (4,0) {$\pnmf$};
\node at (0,3) {$\pnmf$};
\node at (4,3) {$\pnmf$};

\draw[->] (1,0) -- (3,0);
\draw[->] (1,3) -- (3,3);
\draw[->] (0,2.5) -- (0,.5);
\draw[->] (4,2.5) -- (4,.5);

\node at (2,3.2) {\footnotesize $\alpha$};
\node at (2,-.25) {\footnotesize $\beta$};
\node at (-.2,1.5) {\footnotesize $\gamma$};
\node at (4.2,1.5) {\footnotesize $\gamma$};

\end{tikzpicture}
\end{figure}

\noindent In this case $\gamma$ is called a \emph{conjugacy map} or simply a \emph{conjugacy} from $\alpha$ to $\beta$. 

Given an automorphism $\alpha$ of a Boolean algebra $\mathbb A$, the pair $\< \mathbb A,\alpha \>$ is an \emph{algebraic dynamical system}. (Taking Stone duals gives the more commonly studied notion of a topological dynamical system and in the case $\mathbb A = \pnmf$ taking Gelfand--Naimark duals gives a $\bbZ$-action on the unital abelian \cstar-algebra $C(\beta \bbN)$.) Conjugacy is the natural notion of isomorphism in the category of dynamical systems. 
Thus, studying the conjugacy relation on automorphisms of $\pnmf$ amounts simply to studying these maps up to isomorphism, in the appropriate category. 

A \emph{trivial automorphism} of $\pnmf$ is an automorphism induced by the action of some function $\mathbb N \rightarrow \mathbb N$. %For example, the shift map $\s$, defined by $\s([A]_\Fin) = [A+1]_\Fin$
Our goal is to determine when \ch proves two trivial automorphisms are conjugate. 

Let us note from the start that a conjugacy mapping between two trivial automorphisms may be nontrivial. For example, Brian proves in \cite{brian2024does} that assuming \ch, the shift map $\sigma: \pnmf \to \pnmf$ is conjugate to $\sigma^{-1}$. (The shift map, defined by $\s([A]_\mathrm{Fin}) = [A+1]_\mathrm{Fin}$, is the trivial automorphism induced by the successor function $n \mapsto n+1$.) To every trivial automorphism $\alpha$ is associated an integer known as its \emph{prodigal index}, denoted $\Index(\alpha)$ (defined in the next section). 
Two automorphisms with different index can be conjugate: $\sigma$ and $\sigma^{-1}$ provide an example of this, their indices being $1$ and $-1$, respectively. In \cite{van1990automorphism} van Douwen observed that an index obstruction implies that no trivial automorphism can be a conjugacy mapping from $\s$ to $\s^{-1}$. 
Thus \ch proves these maps are conjugate, but in a necessarily nontrivial way.  

%Our expectation was for the picture to look nice and simple under \ch: two trivial automorphisms $\alpha$ and $\beta$ should be conjugate if and only if there is no ``obvious obstruction" to their conjugacy. (See, e.g., Conjecture 2.6(1) in \cite{CoronaRigidity}.) 
%In particular, we expected that $\alpha$ and $\beta$ should be conjugate (under \ch) if the structures $\< \pnmf,\alpha \>$ and $\< \pnmf,\beta \>$ have the same first-order theory, or even just the same existential theory. After all, the existential theory seems to do most of the work for getting a conjugacy between $\sigma$ and $\sigma^{-1}$ in \cite{brian2024does}. 
%
%Instead, we found there is a second-order obstruction to conjugacy, involving the \emph{prodigal index} of the automorphisms. Two automorphisms with different index can be conjugate: $\sigma$ and $\sigma^{-1}$ provide an example of this, their indices being $1$ and $-1$, respectively. But it turns out that two automorphisms cannot be conjugate unless their indices share the same parity.

The main result of Section~\ref{S.Parity} is that, even though automorphisms with different indices can be conjugate, two automorphisms whose indices have different parity cannot be conjugate (this is trivial for automorphisms with finitely many infinite cycles). 
The main result of Section~\ref{S.Saturation} shows that under \ch, this is the only second-order\footnote{To the best of our knowledge; we do not know parity can be expressed as a first-order statement.  See Question~\ref{Q.parity}.} obstruction to conjugacy, resulting in an exact characterization of when \ch proves two trivial automorphisms are conjugate:

\begin{thrm} \label{T.main} For two trivial automorphisms $\alpha$ and $\beta$ of $\pnmf$ the following are equivalent. 
	\begin{enumerate}
		\item \label{1.T.main} \ch{} implies that they are conjugate. 
		\item \label{2.T.main} $\alpha$ and $\beta$ have the same index parity and the structures $\< \pnmf,\alpha \>$ and $\< \pnmf,\beta \>$ are elementarily equivalent. 
		\item \label{3.T.main} $\alpha$ and $\beta$ are conjugate in some forcing extension of the universe. 
	\end{enumerate}
\end{thrm}

The equivalence of the first two conditions is proven in 
Theorem~\ref{thm:ExtensionToTriv}  
and their equivalence with the third one is 
Corollary~\ref{C.Potentially}.

Theorem~\ref{T.main} does not fully settle the issue, however. Rather, it simply raises the question of when two trivial automorphisms are elementarily equivalent. 
Given two maps $f,g: \N \to \N$ that induce trivial automorphisms of $\pnmf$, can we easily decide whether these trivial automorphism have the same theory? Also, we don’t know whether the index parity is expressible in first-order logic.

We prove some results in this direction in Section~\ref{sec:RotaryTheory}. Most notably, we show that there are trivial automorphisms $\alpha$ and $\beta$ of $\pnmf$ such that the structures $\< \pnmf,\alpha \>$ and $\< \pnmf,\beta \>$ are bi-embeddable, but not elementarily equivalent. 
In other words, the class of trivial automorphisms does not satisfy an analog of  the Cantor--Schr\"{o}der--Bernstein theorem. 

Along the way, we give some other model-theoretic results concerning trivial automorphisms. 
For example, we prove $\< \pnmf,\alpha \>$ is countably saturated\footnote{By this we mean $\aleph_1$-saturated, i.e., (since the language is countable) that every countable consistent type is realized.} if and only if the function $\N \to \N$ inducing $\alpha$ does not have infinitely many infinite   
 %$\Z$-like 
orbits. 
%Regarding the shift map, we show that $\< \pnmf,\s \>$ is model complete. 

Even our limited results solve a problem from \cite{braga2018uniformRoe}, belonging to an apparently unrelated topic, a rigidity problem between coarse geometry and operator algebras. On the category of metric spaces one considers the relation of coarse equivalence, where two spaces are identified if their large-scale properties are the same (see~\S\ref{S.uRc} for the definitions). To a metric space one associates a \cstar-algebra, called the uniform Roe algebra, and its quotient, the uniform Roe corona. 
The rigidity question asks whether an isomorphism of algebras implies a coarse equivalence of the original spaces. For uniform Roe algebras the answer is positive: isomorphism (or even Morita-equivalence) of algebras implies coarse equivalence of spaces  (\cite{baudier2022uniform}, see also \cite{martinez2024rigidity} and \cite{martinez2025Cstar} for a generalization to Roe algebras). The isomorphism of uniform Roe coronas implies coarse equivalence of spaces if one assumes forcing axioms (\cite{braga2018uniformRoe}, improved in \cite[Theorem~1.5]{baudier2022uniform}). 
We prove that there are uniformly locally finite metric spaces $X$ and $Y$ such that the isomorphism of the associated uniform Roe algebras is independent from $\mathsf{ZFC}$. More precisely:

\begin{thrm}\label{T.uRq}
There are $2^{\aleph_0}$ uniformly locally finite metric spaces such that 
\begin{enumerate}
	\item They are not coarsely equivalent. 
	\item \ch{} implies that their uniform Roe coronas are isomorphic.
	\item  Forcing axioms\footnote{To be specific, $\OCAT$ and $\MA$.} imply that their uniform Roe coronas are not isomorphic.  
\end{enumerate}  
\end{thrm}

 This is proven in Theorem~\ref{T.uRQ} and in Theorem~\ref{T.uRQ+} we prove that these spaces can be chosen to be of asymptotic dimension 1, and in particular with Property A.

Our results show that \ch{} provides a context in which dynamical systems associated with trivial automorphisms of $\pnmf$ exhibit the least possible degree of rigidity. 
In particular, \ch is the strongest tool one could hope for in trying to build conjugacy mappings between various trivial automorphisms. 
On the other hand, it is consistent that all automorphisms of $\pnmf$ are trivial. 
This was first proved by Shelah \cite{ShelahAutomorphisms}, and later work in \cite{ShelahSteprans}, \cite{Ve:OCA}, and \cite{de2023trivial} revealed that ``every automorphism of $\pnmf$ is trivial'' follows from the axiom $\OCAT$. 
%In models like this, two automorphisms are conjugate if and only if there is a trivial conjugacy mapping between them. 

This phenomenon, that \ch{} implies maximal possible malleability of quotient structures while the forcing axioms provide maximal rigidity is well-known and a survey of the current state of the art can be found in \cite{farah2022corona}. For the effect of \ch{} on automorphisms of $\pnmf$ 
see also \cite[Section 1.6]{vanMill}. 

Readers interested in dynamics but not in uniform Roe algebras (or coronas) can simply stop before they reach \S\ref{S.uRc}. Readers interested in uniform Roe coronas but not in dynamics will have a harder time, because they will have to read Lemma~\ref{L.ReducedProduct}, and some of what leads up to it, in order to make sense of the lemma and its proof, which are relevant to understanding the material in \S\ref{S.uRc}. 

\subsection*{Acknowledgments} 
We would like to thank Saeed Ghasemi for his interest in this project, as well as an unknown referee whose 
comments helped us to clarify several or our proofs.
%pointed out several obscurities in the original draft of this paper. 

\section{Almost permutations}\label{sec:definitions}
%%%%%%%%%%%%

The power set of $\N$, $\P(\N)$, when ordered by inclusion, is a complete Boolean algebra. %The roles of ``join'' and ``meet'' in this algebra are played by the familiar set operations $\cup$ and $\cap$, respectively. 
The Boolean algebra $\pnmf$ 
is the quotient of this Boolean algebra by the ideal of finite sets. 
For $A,B \sub \N$, we write $A =^* B$ to mean that $A$ and $B$ differ by finitely many elements, i.e. $\card{(A \setminus B) \cup (B \setminus A)} < \aleph_0$.
The members of $\pnmf$ are equivalence classes of the $=^*$ relation, and we denote the equivalence class of $A \sub \N$ by $[A]_\text{Fin}$. 
The Boolean-algebraic relation $[A]_\text{Fin} \leq [B]_\text{Fin}$ is denoted on the level of representatives by writing $A \sub^* B$: that is, $A \sub^* B$ means that $[A]_\text{Fin} \leq [B]_\text{Fin}$, or (equivalently) that $B$ contains all but finitely many members of $A$. 

The Stone space of $\P(\N)$ is $\beta\N$, the \v{C}ech-Stone compactification of the countable discrete space $\N$. The Stone space of $\pnmf$ is $\N^* = \beta\N \setminus \N$. 

While our main interest is in almost permutations of $\bbN$, it is helpful to begin with a more general definition. 

\begin{definition}\label{Def.Almost}
An \emph{almost bijection  between sets $X$ and $Y$} is a bijection between a cofinite subset of $X$ and a cofinite subset of $Y$.  
An \emph{almost permutation of a set~$X$} is a bijection between cofinite subsets of $X$. 
By $\fSx(X,Y)$ we denote the set of almost bijections between $X$ and $Y$ and by $\fSx(X)$ we denote the set of almost permutations of $X$. 
Every almost permutation $f$ of $X$ induces an automorphism $\alpha_f$ of $\pxmf$: 
 \[
 \alpha_f ([A]_{\Fin})= [f[A]]_{\Fin}. 
 \] 
An automorphism of $\pxmf$ that is induced in this way by an almost permutation of $X$ is called \emph{trivial}. 

If $f$ and $g$ are almost permutations of $X$ and $f(x) = g(x)$ for all but finitely many $x\in X$, then we write $f =^* g$. 
\hfill\Coffeecup
\end{definition}

If $f$ and $g$ are almost permutations of $\N$, then $g \circ f$ is defined in the natural way, by mapping $n$ to $g(f(n))$ whenever this is well-defined. It is not difficult to see that $g \circ f$ is also an almost permutation, so the class of almost permutations of $\N$ is closed under composition (in the sense just described). Since this composition operator is associative, the structure $(\fSx(X),\circ)$ is  a semigroup with the identity map as a unit. 
Furthermore, for any almost permutations $f$ and $g$ of $X$,
$$\alpha_{f^{-1}} \,=\, \alpha_f^{-1} \qquad \text{ and } \qquad \alpha_{g \circ f} \,=\, \alpha_g \circ \alpha_f.$$
Thus the trivial automorphisms of $\pxmf$ form a group with respect to composition, with identity $\mathrm{id}_{\pxmf} = \alpha_{\mathrm{id}_X}$. This group is denoted $\TrAut(\pxmf)$. 

Because $f\circ f^{-1}=^*\id_X=^*f^{-1}\circ f$, 
the quotient 
\[
\fSs(X) \,=\, \fSx(X)/\!=^*
\] 
is a group as well. 
The $=^*$-equivalence classes of almost bijections (i.e., the elements of $\fSs(X)$) are called  near bijections in \cite[\S 1C]{cornulier2019near}, where the reader can find plenty of information on the structure of  the quotient group $\fSs(X)$. 
This reference is a fun and illuminating read, not  the least because it points out a fair number of incorrect albeit published results. 

%It turns out that the groups $\fSs(X)$ and $\TrAut(\pxmf)$ are interchangeable.  

\begin{lemma}\label{lem:AlmostEqualAPs}
Suppose $f$ and $g$ are almost permutations of $X$. Then $f =^* g$ if and only if $\alpha_f \,=\, \alpha_g$. 
Consequently, 
\[
[f] \mapsto \alpha_f
\]
is a group isomorphism from $\fSs(X)$ onto $\TrAut(\pxmf)$.
\end{lemma}
\begin{proof}
If $f =^* g$, then $f[A] =^* g[A]$ for every $A \sub X$, which means 
\[
\alpha_f ([A]_{\Fin}) = [f[A]]_{\Fin} = [g[A]]_{\Fin} = \alpha_g([A]_{\Fin}).
\]
On the other hand, suppose it is not the case that $f =^* g$, i.e., suppose the set $A=\set{ n \in \dom(f)\cap \dom(g) }{ f(n)\neq g(n) }$ is infinite. Consider the partition $[A]^2$ defined by $c(\{m,n\})=0$ if $f(m)=g(n)$ and $c(\{m,n\})=1$ if $f(m)\neq g(n)$. 
Since both $f$ and $g$ are one-to-one, no 0-homogeneous subset  of $A$ has cardinality greater than $2$. By Ramsey's theorem, there is an infintie 1-homogeneous $B\subseteq X$. But then $f[B]$ and $g[B]$ are disjoint infinite sets, and this implies 
$$\alpha_f ([B]_{\Fin}) = [f[B]]_{\Fin} \neq [g[B]]_{\Fin} = \alpha_g([B]_{\Fin}).$$ 
For the ``consequently'' part: the previous paragraph shows the map $[f] \mapsto \alpha_f$ is bijective, and we have already observed that $\alpha_{f \circ g} = \alpha_f \circ \alpha_g$ and $\alpha_{f^{-1}} = \alpha_f^{-1}$. 
\end{proof}

\begin{definition}\label{def:TriviallyConjugate}
Two automorphisms $\alpha$ and $\beta$ of $\pxmf$ are \emph{conjugate} if there is an automorphism $\gamma$ of $\pxmf$ such that $\gamma \circ \alpha = \beta \circ \gamma$. The map $\gamma$ is called a \emph{conjugacy map} or simply a \emph{conjugacy} from $\alpha$ to $\beta$. 
If some $\gamma \in \TrAut(\pxmf)$ is a conjugacy map from $\alpha$ to $\beta$, then $\alpha$ and $\beta$ are \emph{trivially conjugate}. 
\hfill\Coffeecup
\end{definition}

Note that in the case $\alpha,\beta \in \TrAut(\pxmf)$, the definition of ``trivially conjugate'' reduces to the usual notion of conjugacy in the group $\TrAut(\pxmf)$. 

\begin{definition}\label{def:EquivalentAPs}
Two almost permutations $f$ and $g$  of sets $X$ and $Y$ are called \emph{equivalent}, denoted $f\sim g$,  if there is an almost bijection $h$ between $X$ and $Y$  such that $h \circ f =^* g \circ h$. 
\hfill\Coffeecup
\end{definition}

In the case $Y=X$, Definition~\ref{def:EquivalentAPs} is essentially the same as Definition~\ref{def:TriviallyConjugate}, only at the level of functions rather than equivalence classes. More precisely:

\begin{lemma}\label{lem:TriviallyConjugate}
Suppose $f$ and $g$ are almost permutations of a set $X$. 
Then $f \sim g$ if and only if $\alpha_f$ and $\alpha_g$ are trivially conjugate.  
\end{lemma}
\begin{proof}
Recall that, by Lemma~\ref{lem:AlmostEqualAPs}, the map $[f] \mapsto \alpha_f$ is an isomorphism from $\fSs(X)$ onto $\TrAut(\pxmf)$. 
Thus two almost permutations $f$ and $g$ of $X$ are equivalent if and only if $[f]$ and $[g]$ are conjugate in $\fSs(X)$, if and only if $\alpha_f$ and $\alpha_g$ are conjugate in $\TrAut(\pxmf)$, if and only if $\alpha_f$ and $\alpha_g$ are trivially conjugate.
\end{proof}

\begin{definition}\label{def:FunctionSums}
If $X$ and $Y$ are disjoint sets, $f\in \fSx(X)$, and $g\in \fSx(Y)$, then $f \cup g \in \fSx(X\sqcup Y)$. 
Let $X$ be an infinite set, and let $f,g \in \fSx(X)$. Then $f$ is naturally equivalent to some $f' \in \fSx(X \times \{0\})$, and $g$ is naturally equivalent to some $g' \in \fSx(X \times \{1\})$, and $f' \cup g' \in \fSx(X \times \{0,1\})$. 
But there is a bijection between $X$ and $X \times \{0,1\}$,\footnote{We are assuming the Axiom of Choice, AC,  throughout, but in case of a well-ordered set such as $\bbN$---our main interest---AC is unnecessary.} 
and via this bijection there is a function in $\fSx(X)$ equivalent to $f' \cup g'$. 
All such functions are equivalent to each other (all being equivalent to $f' \cup g'$), so they form a single equivalence class in $\fSs(X)/\sim$. We denote this class by $f \oplus g$, the \emph{direct sum} of $f$ and $g$. 
Abusing notation, we may also sometimes refer to a member of this equivalence class as $f \oplus g$. 
\hfill\Coffeecup
\end{definition}

The $\oplus$ operation turns $\fSs(X)/\sim$ into an abelian semigroup (without a unit).

\begin{definition}\label{Def.sZZ}
Let $s$ denote the successor function on $\N$, that is, $s(n) = n+1$ for all $n \in \N$. The induced automorphism $\alpha_s$ is denoted $\s$, and is known as the (right) \emph{shift map} on $\pnmf$. 
More generally,  for $k\geq 1$ let $s_{k\N}$ denote the direct sum of $k$ copies of $s$. 

Let $s_\Z$ denote a permutation $\N \to \N$ with a single bi-infinite orbit; that is, $s_{\Z}$ is equivalent to the map $n \mapsto n+1$ on $\Z$.  
More generally, let $s_{k\Z}$ denote the sum of $k$ copies of $s_\Z$. 
%Denote $\alpha_{s_{k\Z}} = \alpha_{k\Z}$. 
Let $s_{\Z \times \Z}$ denote a permutation $\N \to \N$ with infinitely many bi-infinite orbits, i.e., $s_{\Z \times \Z}$ is equivalent to the map $(i,j) \mapsto (i,j+1)$ on $\Z \times \Z$. 
Denote $\alpha_{s_{\Z \times \Z}} = \alpha_{\Z \times \Z}$. 
\hfill\Coffeecup
\end{definition}

Observe that $s_{k\N}$ is equivalent to $s^k$ for every $k > 0$. Our notation, $s_{k\N}$ rather than $s^k$, is meant to encourage thinking of $s_{k\N}$ as consisting of $k$ $\N$-like orbits. 

\begin{definition}\label{def:RotaryMap}
Let $\bar n = \seq{n_j}{j \in \w}$ be a sequence in $\N$. 
Define $I_0 = [0,n_0)$, $I_1 = [n_0,n_0+n_1)$, and generally, $I_k = \big[ \sum_{i = 0}^{k-1}n_i , \sum_{i = 0}^{k}n_i \big)$ for $k > 0$, so that each $I_k$ is the interval of length $n_k$ immediately to the right of $I_{k-1}$. Let 
$$r_{\bar n}(i) \,=\, \begin{cases}
i+1 &\text{ if } i \neq \max I_k \text{ for any $k$,} \\
\min I_k &\text{ if } i = \max I_k.
\end{cases}$$
In other words, $r_{\bar n}$ acts on each $I_k$ by cyclically permuting the members of $I_k$. 
Any permutation of $\N$ of this form is called a \emph{rotary permutation} on $\N$. 

For convenience, we denote $\alpha_{r_{\bar n}}$ by $\alpha_{\bar n}$. Any such trivial automorphism of $\pnmf$ is called a \emph{rotary automorphism}. 
\hfill\Coffeecup
\end{definition}

In the following definition we specialize to $\fSx(\bbN)$, although it is not difficult to see that the definition transfers almost verbatim to $\fSx(X)$ for any infinite set $X$. 

\begin{definition}\label{def:OrbitStructure}
Let $f$ be an almost permutation of $\N$. Given $n \in \mathrm{dom}(f)$, let $O_f(n) = \set{f^k(n)}{k \in \Z \text{ and $f^k(n)$ is well-defined}}$, the orbit of $n$ under $f$. 
We say that $O_f(n)$ is \emph{$\N$-like} if there is some $k < 0$ such that $f^k(n) \notin \mathrm{image}(f)$; we say $O_f(n)$ is \emph{reverse $\N$-like} if there is some $k > 0$ such that $f^k(n) \notin \mathrm{dom}(f)$; we say $O_f(n)$ is \emph{$\Z$-like} if it is infinite and $f^k(n)$ is well-defined for all $k \in \Z$. 

If $f$ has infinitely many finite cyclic orbits, fix an enumeration $\bar n = \seq{n_j}{j \in \w}$ of the sizes of these orbits. The \emph{rotary part} of $f$ is 
$\mathrm{R}(f) \,=\, r_{\bar n}$. 
If $f$ has only finitely many finite orbits, we say the rotary part of $f$ is empty, and $\mathrm{R}(f)$ is defined to be the empty map. 

Let $N_+$ denote the number of $\N$-like orbits of $f$, and let $N_-$ the number of reverse $\N$-like orbits. Both these numbers are finite, because $f$ is a bijection between cofinite subsets of $\N$. 
Let $k = N_+ - N_-$. 
If $k > 0$, the \emph{$\N$-like part} of $f$ is 
$\mathrm{S}(f) \,=\, s_{k\N}$. 
If $k < 0$, the \emph{$\N$-like part} of $f$ is 
$\mathrm{S}(f) \,=\, s_{k\N}^{-1}$. 
If $k = 0$, we say the \emph{$\N$-like part} of $f$ is empty, and $\mathrm{S}(f)$ is defined to be the empty map. 

If $f$ has infinitely many $\Z$-like orbits, the \emph{$\Z$-like part} of $f$ is 
$\mathrm{Z}(f) \,=\, s_{\Z \times \Z}$. 
If $f$ has some finite number $\ell$ of $\Z$-like orbits, let $m = \min\{ N_+,N_- \}$ and let $k = \ell+m$. 
If $k > 0$, the \emph{$\Z$-like part} of $f$ is 
$\mathrm{Z}(f) \,=\, s_{k\Z}$. 
If $k = 0$, we say the $\Z$-like part of $f$ is empty, and $\mathrm{Z}(f)$ is defined to be the empty map. 
\hfill\Coffeecup
\end{definition}

Note that $\mathrm R(f)$ is only defined up to equivalence, as it depends on the enumeration $\bar n$. 
In what follows (in the next lemma and in the remainder of the paper), every statement about $\mathrm R(f)$ is invariant under equivalence.

\begin{lemma}\label{lem:OrbitStructure}
Every almost permutation $f$ of $\bbN$ is equivalent to $\mathrm{R}(f) \oplus \mathrm{Z}(f) \oplus \mathrm{S}(f)$. 
Moreover, given two almost permutations $f$ and $g$ of $\N$, $f \sim g$ if and only if $\mathrm R(f) \sim \mathrm R(g)$, $\mathrm S(f) = \mathrm S(g)$, and $\mathrm Z(f) = \mathrm Z(g)$. 
\end{lemma}

\begin{proof}
If $O_f(n)$ is $\N$-like and $O_f(m)$ is reverse $\N$-like, then the restriction of $f$ to $O_f(n) \cup O_f(m)$   is equivalent to $s_\Z$, a single $\Z$-like orbit. (To see this, extend $f$ to $O_f(m)\cup O_f(n)$ by sending the unique element of $f[O_f(n)]\setminus O_f(n)$  to the unique element of $O_f(m)\setminus f[O_f(m)]$.) 
Bearing this in mind, the desired conclusion is an almost immediate consequence of Definition~\ref{def:OrbitStructure}, as follows. 
First, decompose $f$ into orbits. 
Because $f$ is an almost permutation of $\N$, it can have only finitely many finite orbits that are not finite cycles. Delete these (if any exist), 
and delete the other finite orbits as well if there are only finitely many of them. 
Next, exchange any pair of $\N$-like and reverse $\N$-like orbits for a single $\Z$-like orbit. 
These changes result in a map equivalent to $f$, and this map is equivalent to $\mathrm{R}(f) \oplus \mathrm{Z}(f) \oplus \mathrm{S}(f)$. 

To prove the ``moreover"  part, we first consider the case when $g$ is the restriction of $f$ to a cofinite set. Let $S \subseteq \bbN$ be the complement of this set. Then $S$ intersects only finitely many orbits of $f$, and by induction it suffices to consider the case when it intersects only one orbit, $O_f(n)$. The only nontrivial case is when this is a two-sided infinite orbit, in which case passing from $f$ to $g$ breaks it into two one-sided infinite orbits, one $\N$-like and the other reverse $\N$-like, together with finitely many (possibly zero) finite orbits. It is a feature of our definition that these two infinite orbits contribute to $\mathrm S(g)$ and $\mathrm Z(g)$ in the same way that the single $\Z$-like orbit contributes to $\mathrm S(f)$ and $\mathrm Z(f)$. This proves the special case. 

For the general case, fix equivalent $f$ and $g$ and an almost bijection $h$ witnessing the equivalence. Let $f'=f\rs\dom(h)$ and $g'=g\rs \ran(h)$. By the previous paragraph, $f\sim f'$ and $g\sim g'$. As $h$ witnesses $f'\sim g'$, this concludes the proof. 
\end{proof}

\begin{definition}
If $f$ is an almost permutation of $\N$, define
\[
\Index(f) \,=\,
|\bbN\setminus \dom(f)|-|\bbN\setminus \ran(f)|
\]
This is the \emph{index} of $f$, or more precisely the \emph{prodigal index} of $f$. 
The (\emph{index}) \emph{parity} of $f$, denoted $\parity(f)$, is $0$ if $\Index(f)$ is even, and $1$ if $\Index(f)$ is odd. 
\hfill\Coffeecup
\end{definition}

The  index first appeared in a rudimentary form in a 1915 theorem of Andreoli (see \cite[(12) on p. 25]{cornulier2019near}). 
The quantity $-\Index(f)$ is called the \emph{banker index} of $f$; see the discussion in \cite[\S 1.C]{cornulier2019near}. The two indices are clearly interchangeable, and our choice of prodigal index over banker index is dictated by the analogy with the Fredholm index of operators on a Hilbert space (see e.g., \cite[3.3.12]{Pede:Analysis}). 

As is well-known, the Fredholm index of Fredholm operators on a Hilbert space is preserved under compact perturbations (\cite[Theorem~3.3.17]{Pede:Analysis}). Similarly, the index of almost permutations is preserved under the $=^*$ relation. To prove this, it suffices to notice that every almost permutation $f$ has the same index as $f \restriction A$ for any cofinite $A \subseteq \mathrm{dom}(f)$; this implies that if $f =^* g$ then $f$ and $g$ both have the same index as $f \cap g$, hence the same index as each other. 
Similar reasoning shows that the index of an almost permutation is invariant under equivalence as well: i.e., $f \sim g$ implies $\Index(f) = \Index(g)$. 

Furthermore, $\Index(f\circ g)=\Index(f)+\Index(g)$ (see \cite[Theorem 6.1]{van1990automorphism}). 
Thus the index is a group homomorphism from $\fSs(\bbN)$ to $\bbZ$ whose kernel is $\fS_0^*(\bbN)$, the group of permutations of $\bbN$ modulo almost equality. 
It follows that the parity function is a group homomorphism from $\fSs(\bbN)$ to $\bbZ/2\bbZ$ whose kernel is the group of parity zero almost permutations.

In light of Lemma~\ref{lem:AlmostEqualAPs}, and the fact that $\Index(f) = \Index(g)$ whenever $f =^* g$, we can define
\[
\Index(\alpha_f) \,=\, \Index(f) \qquad \text{ and } \qquad \parity(\alpha_f) \,=\, \parity(f).
\]
Just like the Fredholm index of operators on an infinite-dimensional, separable Hilbert space, the prodigal index is a group homomorphism from the group of trivial automorphisms of $\pnmf$ into $(\bbZ,+)$; in other words, 
\[
\Index(\alpha_f \circ \alpha_g) \,=\, \Index(\alpha_f) + \Index(\alpha_g).
\] 
(This follows from the facts that $[f] \mapsto \alpha_f$ is a group isomorphism from $\fSs(X)$ to $\TrAut(\pxmf)$, and the index is a group homomorphism from $\fSs(\bbN)$ to $\bbZ$. See \cite[Section 6]{van1990automorphism} for further details.) 
In particular, this proves: 

\begin{proposition}\label{prop:IndexObstruction}
If two trivial automorphisms of $\pnmf$ are trivially conjugate, then they have the same index. \qed
\end{proposition}

 Similarly, from this and Lemma~\ref{lem:TriviallyConjugate}, it follows that equivalent almost permutations always have the same index. 

Because equivalent permutations have the same index, Lemma~\ref{lem:OrbitStructure} implies 
\begin{align*}
\Index(f) &\,=\, \Index(\rmR(f) \oplus \rmZ(f) \oplus \rmS(f)) \\
&\,=\, \Index(\rmR(f))+\Index(\rmZ(f))+\Index(\rmS(f)) \,=\, \Index(\mathrm S(f)).
\end{align*}
In other words, $\Index(f)$ depends only on $\mathrm{S}(f)$. 
Specifically, 
\begin{align*}
\Index(f) &= k > 0 \quad \Leftrightarrow \quad \mathrm{S}(f) = s_{k\N}, \\
\Index(f) &= k < 0 \quad \Leftrightarrow \quad \mathrm{S}(f) = s_{|k|\N}^{-1}, \\
\Index(f) &= 0 \hspace{10.1mm} \Leftrightarrow \quad \mathrm{S}(f) \text{ is empty}. 
\end{align*} 

One could paraphrase Proposition~\ref{prop:IndexObstruction} as saying that the index provides an ``obstruction'' to trivial conjugacy: automorphisms with different indices are not trivially conjugate. 
Let us point out, though, that the index is not an obstruction to conjugacy in general. This was proved recently by Brian \cite{brian2024does}, by showing that \ch implies $\s$ and $\s^{-1}$ are conjugate. (Note that these maps have index $1$ and $-1$, respectively.) 
%The property of having a trivial index is also not generally preserved by conjugacy: since \ch implies $\s$ and $\s^{-1}$ are conjugate, \ch also implies that $\sigma \oplus \sigma^{-1}$ and $\sigma \oplus \sigma$ are conjugate, in spite of their indices being $0$ and $2$. 
In the following section we show that while the index does not provide an obstruction to conjugacy, the parity of the index does.

%%%%%%%%%%%%
\section{The parity obstruction}
%%%%%%%%%%%%
 \label{S.Parity}

The goal of this section is to prove that conjugate trivial automorphisms of $\pnmf$ have the same index parity. Note that this is a theorem of $\mathsf{ZFC}$, and does not require \ch. 

\begin{definition}\label{def:Components}
Let $\alpha$ be an automorphism of $\pnmf$. Define 
$$\mathrm{FIX}_\alpha = \set{c \in \pnmf}{\alpha(c)=c}.$$
$\FIX_\alpha$ is a Boolean subalgebra of $\pnmf$, and the atoms of $\FIX_\alpha$ are called the \emph{components} of $\alpha$. That is, $c$ is a component of $\alpha$ if $c \in \mathrm{FIX}_\alpha \setminus \{[\emptyset]_\mathrm{Fin}\}$ and if $x \notin \FIX_\alpha$ whenever $[\emptyset]_\mathrm{Fin} < x < c$. 
\hfill\Coffeecup
\end{definition}

Observe that $\FIX_\s$ and $\FIX_{\s^{-1}}$ are both the trivial algebra $\{[\emptyset]_\mathrm{Fin},[\N]_\mathrm{Fin}\}$. 
In fact, $\s$ and $\s^{-1}$ both enjoy a stronger property: they are \emph{incompressible}, which means that if $x \in \pnmf$ and $x \neq [\emptyset]_\mathrm{Fin},[\N]_\mathrm{Fin}$, then $\s(x),\s^{-1}(x) \not\leq x$. 
This is not difficult to see; a proof can be found in \cite[Theorem 3.2]{brian2024does} or \cite[Theorem 5.3]{brianPsets}. 

Furthermore, $\s$ and $\s^{-1}$ are the only trivial automorphisms of $\pnmf$ for which $\FIX_\alpha = \{[\emptyset]_\mathrm{Fin},[\N]_\mathrm{Fin}\}$. 
(It is worth noting, however, that \ch implies there are non-trivial automorphisms $\alpha$ with $\FIX_\alpha$ the trivial subalgebra; see \cite[Theorem 5.7]{brianPsets}.) 
With this in mind, the idea behind Definition~\ref{def:Components} is that when $\alpha$ is trivial, each component of $\alpha$ is simply the clopen set corresponding to some $\N$-like or reverse $\N$-like orbit in the almost permutation inducing $\alpha$. This is made more precise in Lemmas~\ref{lem:FIX2} and \ref{lem:FIX3} below. 

\begin{lemma} \label{lem:FIX}
Let $\alpha$ and $\beta$ be automorphisms of $\pnmf$. If $\gamma$ is a conjugacy map from $\alpha$ to $\beta$, then $\FIX_\beta = \{ \gamma(c) \mid\, c \in \FIX_\alpha\}$. In other words, conjugacy mappings send fixed points to fixed points. Furthermore, $c$ is a component of $\alpha$ if and only if $\gamma(c)$ is a component of $\beta$.
\end{lemma}
\begin{proof} 
If $\gamma$ is a conjugacy mapping from $\alpha$ to $\beta$, then $\alpha(c) = \gamma^{-1} \circ \beta \circ \gamma(c)$ for all $c \in \pnmf$. Because $\gamma$ is injective, $\alpha(c) = c$ if and only if $\beta(\gamma(c)) = \gamma(c)$.

For the ``furthermore'' part, note that $\gamma$, being an automorphism, preserves the~$\leq$ relation. Thus $c$ is an atom of $\FIX_\alpha$ if and only if $\gamma(c)$ is an atom of $\FIX_\beta$.
\end{proof}

Given an almost permutation $f$ of $\N$ and some $n \in \N$ with a $\Z$-like orbit, let us denote $O^+_f(n) = \set{f^i(n)}{i > 0}$ and $O^-_f(n) = \set{f^i(n)}{i < 0}$. 

\begin{lemma} \label{lem:FIX2}
Let $f$ be an almost permutation of $\N$. Let $O$ be either an $\N$-like orbit, or a reverse $\N$-like orbit, or a set of the form $O^+_f(n)$ or $O^-_f(n)$ for some $n$ with a $\Z$-like orbit. 
If $A \subseteq \N$ and $O \cap A$ and $O \setminus A$ are both infinite, then $[A]_\mathrm{Fin} \notin \FIX_{\alpha_f}$. 
\end{lemma}
\begin{proof} 
Suppose that $n$ has either an $\N$-like orbit or a reverse $\N$-like orbit, and suppose that $O_f(n) \cap A$ and $O_f(n) \setminus A$ are both infinite. Let 
$$B = \set{f^i(n) \in O_f(n)}{f^i(n) \in A \text{ and } f^{i+1}(n) \notin A}.$$
Because $O_f(n) \cap A$ and $O_f(n) \setminus A$ are both infinite, $B$ is infinite. 
But $B\subseteq A$, while $\alpha_f([B]_\mathrm{Fin}) = [f[B]]_\mathrm{Fin} \not\leq [A]_\mathrm{Fin}$. 
Hence $\alpha_f([A]_\mathrm{Fin}) \neq [A]_\mathrm{Fin}$. 

If $n$ has a $\Z$-like orbit, the same argument works, verbatim, if we simply replace $O_f(n)$ with either $O^+_f(n)$ or $O^-_f(n)$.
\end{proof}

\begin{lemma} \label{lem:FIX3}
Let $f$ be an almost permutation of $\N$. Then components of $\alpha_f$ are precisely the sets $[O]_\mathrm{Fin}$ where $O$ be either an $\N$-like orbit, or a reverse $\N$-like orbit, or a set of the form $O^+_f(n)$ or $O^-_f(n)$ for some $n$ with a $\Z$-like orbit. 
\end{lemma}
\begin{proof} 
Let $O$ be either an $\N$-like orbit, or a reverse $\N$-like orbit, or a set of the form $O^+_f(n)$ or $O^-_f(n)$ for some $n$ with a $\Z$-like orbit. Clearly $\alpha_f([O]_\mathrm{Fin}) = [O]_\mathrm{Fin}$, i.e., $[O]_\mathrm{Fin} \in \FIX_{\alpha_f}$. Furthermore, the previous lemma implies that if $A \sub \N$ and $[\emptyset]_\mathrm{Fin} < [A]_\mathrm{Fin} < [O]_\mathrm{Fin}$, then $[A]_\mathrm{Fin} \notin \FIX_{\alpha_f}$. Thus $[O]_\mathrm{Fin}$ is a component of $\alpha_f$. 

It remains to show $\alpha_f$ has no other components. 
Fix $A \subseteq \N$.% and suppose $[A]_\mathrm{Fin}$ is a component of $\alpha_f$. 

Let $O$ be either an $\N$-like orbit, or a reverse $\N$-like orbit, or a set of the form $O^+_f(n)$ or $O^-_f(n)$ for some $n$ with a $\Z$-like orbit. 
If $[O]_\mathrm{Fin} < [A]_\mathrm{Fin}$ then (because $[O]_\mathrm{Fin} \in \FIX_{\alpha_f}$) $[A]_\mathrm{Fin}$ is not a component of $\alpha_f$. 
If $[\emptyset]_\mathrm{Fin} < [O]_\mathrm{Fin} \wedge [A]_\mathrm{Fin} < [A]_\mathrm{Fin}$, then $[A]_\mathrm{Fin}$ is not a component of $\alpha_f$ by the previous lemma.  
Hence if $[A]_\mathrm{Fin}$ is a component of $\alpha_f$, we must have either $[O]_\mathrm{Fin} = [A]_\mathrm{Fin}$ or $[O]_\mathrm{Fin} \cap [A]_\mathrm{Fin} = [\emptyset]_\mathrm{Fin}$.  

Suppose, then, that $A \cap O$ is finite whenever $O$ is an $\N$-like orbit, or a reverse $\N$-like orbit, or a set of the form $O^+_f(n)$ or $O^-_f(n)$ for some $n$ with a $\Z$-like orbit. 
This implies $A$ meets every orbit in a finite set. So either $A$ is finite or it meets infinitely many orbits. If $A$ is finite, $[A]_\mathrm{Fin} = [\emptyset]_\mathrm{Fin}$ is not a component of $\alpha_f$. If $A$ meets infinitely many orbits, then there is a partition of $A$ into two infinite subsets $A_0$ and $A_1$ such that if $n \in A_i$ then $O_f(n) \cap A \sub A_i$. If $\alpha_f([A]_\mathrm{Fin}) = [A]_\mathrm{Fin}$, then this implies $\alpha_f([A_i]_\mathrm{Fin}) = [A_i]_\mathrm{Fin}$ for $i=0,1$. Consequently, either $[A]_\mathrm{Fin} \notin \FIX_{\alpha_f}$, or else $[A_0]_\mathrm{Fin},[A_1]_\mathrm{Fin} \in \FIX_{\alpha_f}$. Either way, $[A]_\mathrm{Fin}$ is not a component.
\end{proof} 

Recall from Lemma~\ref{lem:OrbitStructure} that every almost permutation $f$ of $\N$ is equivalent to $\mathrm R(f) \oplus \mathrm S(f) \oplus \mathrm Z(f)$. If $f = \mathrm R(f) \oplus \mathrm S(f) \oplus \mathrm Z(f)$ then, by the definition of the direct sum (Definition~\ref{def:FunctionSums}), the domain of $\mathrm R(f)$ is the subset of $\N$ containing those $n \in \N$ with finite $f$-orbits.

\begin{lemma} \label{lem:DefiningTheRotary}
Let $f$ be an almost permutation of $\N$ with $f = \mathrm R(f) \oplus \mathrm S(f) \oplus \mathrm Z(f)$, and let $D = \mathrm{dom}(\mathrm{R}(f))$. Then $[\N \setminus D]_{\Fin}$ is the supremum of all components of $\alpha_f$.
\end{lemma}
\begin{proof} 
This follows immediately from the previous lemma.
\end{proof}

\begin{theorem}\label{thm:Fredholm}
If two trivial automorphisms of $\pnmf$ have different  index parity, then they are not conjugate. 
\end{theorem}
\begin{proof}
Consider the following property of an automorphism $\alpha$ of $\pnmf$: 
\begin{itemize}
\item[$(*)$] 
There is a partition $\P$ of the components of $\alpha$ into pairs, such that for any $A \sub \P$ (any set of these pairs), there is some $d_A \in \pnmf$ such that
\begin{itemize}
\item[$\circ$] If $c$ is a component of $\alpha$ and $c \in \bigcup A$, then $c \leq d_A$.
\item[$\circ$] If $c$ is a component of $\alpha$ and $c \notin \bigcup A$, then $c \wedge d_A = [\emptyset]_\mathrm{Fin}$.
\item[$\circ$] $\alpha(d_A) = d_A$. 
\end{itemize}
\end{itemize}

This property $(*)$ is conjugacy-invariant. 
One can see this simply by noting that $(*)$ describes a property of the algebraic dynamical system $\< \pnmf,\alpha \>$ that does not depend on the names of members of $\pnmf$; as conjugacy is the natural notion of isomorphism in the category of dynamical systems, any such property is conjugacy-invariant. 
For a more hands-on way of seeing that $(*)$ is conjugacy-invariant, we noted already in Lemma~\ref{lem:FIX} that a conjugacy mapping $\gamma$ sends the components of $\alpha$ to the components of a conjugate map $\beta = \phi \circ \alpha \circ \phi^{-1}$, in the sense that $c$ is a component of $\alpha$ if and only if $\gamma(c)$ is a component of $\beta$. 
Using this fact, it is not difficult to check that if $\P$ is a partition witnessing $(*)$ for $\alpha$, then $\set{\{\gamma(c),\gamma(d)\}}{\{c,d\} \in \P}$ is a partition witnessing $(*)$ for $\gamma \circ \alpha \circ \gamma^{-1}$. 

The plan of this proof is to show that property $(*)$ detects the  index parity of a trivial map, in the sense that if $f$ is an almost permutation of $\N$, then $\alpha_f$ satisfies~$(*)$ if and only if $\parity(f) = 0$. Because $(*)$ is conjugacy-invariant, this suffices to prove the theorem.

%\vspace{1.5mm}

First let us show that if $\parity(f) = 0$ then $\alpha_f$ satisfies $(*)$. 
Let $f$ be an almost permutation of $\N$ with $\parity(f) = 0$. 
Recall that $\alpha_f$ is trivially conjugate to $\alpha_{f'}$ for any almost permutation $f'$ of $\N$ that is equivalent to $f$. 
Because $(*)$ is conjugacy-invariant, this means that replacing $f$ with an equivalent almost permutation does not change whether it satisfies $(*)$. 
Thus, applying Lemma~\ref{lem:OrbitStructure}, we may (and do) assume $f = \mathrm{R}(f) \oplus \mathrm{Z}(f) \oplus \mathrm{S}(f)$. 
By the remarks near the end of Section~\ref{sec:definitions}, the index of $f$ depends only on $\mathrm S(f)$. 
In particular, because $\parity(f) = 0$, either $\mathrm S(f) = s_{k\N}$ or $\mathrm S(f) = s_{k\N}^{-1}$ for some even $k \geq 0$. 

Suppose first that $f$ has only finitely many $\Z$-like orbits. 
By Lemma~\ref{lem:FIX3}, $\alpha_f$ has two components for every $\Z$-like orbit of $f$ and one for every $\N$-like or reverse $\N$-like orbit. 
In particular, $\alpha_f$ has a finite, even number of components. 
In this case, any partition $\P$ of these components into pairs witnesses that $\alpha_f$ satisfies $(*)$, because if $A \sub \P$ then $d_A = \bigvee \big( \bigcup A \big)$ has all the required properties. 

Suppose next that $f$ has infinitely many $\Z$-like orbits. 
Let $\P$ denote a partition of the components of $\alpha_f$ such that
\begin{itemize}
\item[$\circ$] if $O$ is an $\N$-like or reverse $\N$-like orbit of $f$, then there is another $\N$-like or reverse $\N$-like orbit $O'$ such that $\{ [O]_\mathrm{Fin},[O']_\mathrm{Fin} \} \in \P$; and 
\item[$\circ$] for each $\Z$-like orbit $O_f(n)$, $\{ [O^-_f(n)]_\mathrm{Fin},[O^+_f(n)]_\mathrm{Fin}, \} \in \P$.
\end{itemize}
In other words, the components of $\alpha_f$ that come from $\Z$-like orbits are paired up in the natural way, and the components coming from $\N$-like or reverse $\N$-like orbits are paired with each other arbitrarily. This covers all components of $\alpha_f$ by Lemma~\ref{lem:FIX3}. 
To see that $\P$ witnesses $(*)$, fix $A \sub \P$. 
For each pair $p \in A$ either $(1)$ there are two $\N$-like or reverse $\N$-like orbits from $f$ giving rise to the two components in $p$, or $(2)$ the two components of $p$ are formed from the two ends of a single $\Z$-like orbit. 
Let $D_A$ denote the union of all these orbits: i.e., 
$$\textstyle D_A \,=\, \bigcup \set{O_f(n)}{[O_f(n)]_{\Fin} \in \bigcup A \text{ or } [O^-_f(n)]_{\Fin}, [O^+_f(n)]_{\Fin} \in \bigcup A }.$$ 
Using Lemmas~\ref{lem:FIX3} and the assumption that $f = \mathrm{R}(f) \oplus \mathrm{Z}(f) \oplus \mathrm{S}(f)$, it is not difficult to see that $d_A = [D_A]_{\Fin}$ has the three properties required in $(*)$. 

Hence, whether $f$ has finitely many or infinitely many $\Z$-like orbits, if $\parity(f) = 0$ then $\alpha_f$ satisfies $(*)$. 

Next, let us show that if $\parity(f) = 1$ then $\alpha_f$ does not satisfy $(*)$. 
Let $f$ be an almost permutation of $\N$ with $\parity(f) = 1$. 
As before, we may (and do) assume $f = \mathrm{R}(f) \oplus \mathrm{Z}(f) \oplus \mathrm{S}(f)$. 
Because $\parity(f) = 1$, either $\mathrm S(f) = s_{k\N}$ or $\mathrm S(f) = s_{k\N}^{-1}$ for some odd $k > 0$. 

By Lemma~\ref{lem:FIX3}, $\alpha_f$ has two components for every $\Z$-like orbit of $f$ and one for every $\N$-like or reverse $\N$-like orbit. 
But $f$ has an odd number of $\N$-like or reverse $\N$-like orbits, so if $f$ has only finitely many $\Z$-like orbits, then $\alpha_f$ has a finite, odd number of components. 
In particular, the components of $\alpha_f$ cannot be partitioned into pairs, which means that $(*)$ fails for trivial reasons. 

Suppose instead that $f$ has infinitely many $\Z$-like orbits. 
Because $\parity(f) = 1$, the number of $\N$-like or reverse $\N$-like orbits of $f$ is odd. 
Thus $\alpha_f$ has $\aleph_0$ components coming from the $\Z$-like orbits of $f$, and it has a finite odd number of components coming from the $\N$-like and reverse $\N$-like orbits of $f$. 

Each $\Z$-like orbit $O_f(n)$ gives rise to two components of $\alpha_f$, $[O^-_f(n)]_{\Fin}$ and $[O^+_f(n)]_{\Fin}$. Let us say that these two components are \emph{natural partners}. Furthermore, let us say that any component of $\alpha_f$ arising from $\N$-like or reverse $\N$-like orbit does not have a natural partner.

Suppose $\P$ is a partition of the components of $\alpha_f$ into pairs. 
Because the number of $\N$-like and reverse $\N$-like orbits of $f$ is odd, the number of components without a natural partner is odd. 
This implies that there is an infinite $A_0 \sub \P$ such that each member of $A_0$ has a natural partner, but is not paired with its natural partner by $\P$. 
(To see this, note that if $\P$ were to pair all but finitely many components with their natural partners, then there would be a finite, odd number of components left over, and so no way to pair them off.) 
If we color $[A_0]^2$ by setting $\chi(p,p') = 1$ if some member of $p$ is the natural partner of some member of $p'$, and $\chi(p,p') = 0$ otherwise, then clearly there is no infinite $1$-homogeneous subset of $A_0$. 
Thus, by Ramsey's theorem, this coloring has an infinite $0$-homogeneous set $A \sub A_0$. 
That is, there is an infinite $A \sub A_0$ such that each member of $\bigcup A$ has a natural partner, but no two members of $\bigcup A$ are natural partners of each other. 

Concretely, there is an infinite set $\set{n_k}{k \in \w} \sub \N$ such that the $O_f(n_k)$ are all distinct $\Z$-like orbits of $f$ (hence disjoint from each other), and such that $\bigcup A$ contains exactly one of $[O^-_f(n_k)]_{\Fin}$ or $[O^+_f(n_k)]_{\Fin}$ for each $k$.

Fix $d_A \in \pnmf$, and fix $D \sub \w$ with $[D]_{\Fin} = d_A$. 
Assume that $d_A$ satisfies the first two of the three properties listed in the statement of $(*)$.  
We will show that $d_A$ cannot also satisfy the third property. 

Our assumption that $d_A$ satisfies the first two properties in $(*)$ means that for any component $c$ of $\alpha_f$, either $c \in \bigcup A$ and $c \leq d_A$, or else $c \notin \bigcup A$ and $c \wedge d_A = [\emptyset]_\mathrm{Fin}$. 
Given our description of the components of $\alpha_f$ above, this means that for each $k \in \w$, either $D$ contains only finitely many members $O^-_f(n_k)$ and all but finitely many members of or $O^+_f(n_k)$, or else $D$ contains only finitely many members $O^+_f(n_k)$ and all but finitely many members of or $O^-_f(n_k)$. 
This implies that, for each $k$, there is some $i_k \in \Z$ such that either $f^{i_k}(n_k) \in D$ and $f^{i_k+1}(n_k) \notin D$, or else $f^{i_k}(n_k) \notin D$ and $f^{i_k+1}(n_k) \in D$. 

If $f^{i_k}(n_k) \in D$ and $f^{i_k+1}(n_k) \notin D$ for infinitely many $k$, then letting $K$ denote the set of all such $k$, we have 
$$[\set{f^{i_k+1}(n_k)}{k \in K}]_{\Fin} \,=\, \alpha_f([\set{f^{i_k}(n_k)}{k \in K}]_{\Fin}) \,\leq\, \alpha_f([D]_{\Fin}) \,=\, \alpha_f(d_A),$$
while at the same time 
\begin{align*}
[\set{f^{i_k+1}(n_k)}{k \in K}]_{\Fin} \wedge d_A &\,=\, [\set{f^{i_k+1}(n_k)}{k \in K}]_{\Fin} \wedge [D]_{\Fin} \\
 &\,=\, [\set{f^{i_k+1}(n_k)}{k \in K} \cap D]_{\Fin} \,=\, [\emptyset]_{\mathrm{Fin}}.
\end{align*}
Because $K$ is infinite, $[\set{f^{i_k+1}(n_k)}{k \in K}]_{\Fin} \neq [\emptyset]_\mathrm{Fin}$, and therefore $\alpha_f(d_A) \neq d_A$. 
If on the other hand $f^{i_k}(n_k) \notin D$ and $f^{i_k+1}(n_k) \in D$ for infinitely many $k$, then letting $K$ denote the set of all such $k$, we have 
$$[\set{f^{i_k}(n_k)}{k \in K}]_{\Fin} = \alpha_f^{-1}([\set{f^{i_k+1}(n_k)}{k \in K}]_{\Fin}) \leq \alpha_f^{-1}([D]_{\Fin}) = \alpha_f^{-1}(d_A),$$
while at the same time 
\begin{align*}
[\set{f^{i_k}(n_k)}{k \in K}]_{\Fin} \wedge d_A &\,=\, [\set{f^{i_k}(n_k)}{k \in K}]_{\Fin} \wedge [D]_{\Fin} \\
 &\,=\, [\set{f^{i_k}(n_k)}{k \in K} \cap D]_{\Fin} \,=\, [\emptyset]_\mathrm{Fin}. 
\end{align*}
As before, $[\set{f^{i_k}(n_k)}{k \in K}]_{\Fin} \neq [\emptyset]_\mathrm{Fin}$, so this implies $\alpha_f^{-1}(d_A) \neq d_A$. 
In either case, $\alpha_f(d_A) \neq d_A$. 
This shows that any $d_A \in \pnmf$ satisfying the first two bullet points listed in statement $(*)$ fails to satisfy the third. In particular, no $d_A$ satisfies all three at the same time, and therefore $(*)$ fails for $\alpha_f$. 
\end{proof}

Recall from the previous section that $\alpha \mapsto \Index(\alpha)$ is a group homomorphism $\TrAut(\pnmf) \to \Z$. It follows, of course, that $\alpha \mapsto \parity(\alpha)$ is a group homomorphism $\TrAut(\pnmf) \to \Z/2\Z$.

Now observe that the property $(*)$ from the preceding proof is a property of automorphisms, not merely of trivial automorphisms or of almost permutations. 
In particular, one may wonder whether $(*)$ defines a coherent notion of ``even'' and ``odd'' automorphisms generally (not just for trivial automorphisms). In other words, does $(*)$ define a group homomorphism $\Aut(\pnmf) \to \Z/2\Z$?

The answer is a resounding no. Not only does $(*)$ fail to do this, but in fact \ch implies there is no homomorphism from $\Aut(\pnmf)$ to $\Z/2\Z$, or to any other group other than the trivial group or $\Aut(\pnmf)$ itself. 
In other words, \ch implies $\Aut(\pnmf)$ is simple. Building on results of Anderson \cite{Anderson}, this was proved independently by Fuchino \cite{Fuchino} and Rubin \cite{RubinStepanek}. 

Theorem~\ref{thm:Fredholm} constitutes part of the proof that \eqref{1.T.main} $ \Rightarrow$ \eqref{2.T.main} in Theorem~\ref{T.main} from the introduction. The rest of the proof of Theorem~\ref{T.main}, which relies on \ch via model-theoretic methods, is proved in the following section.

%%%%%%%%%%%%
\section{Saturation}
%%%%%%%%%%%%
\label{S.Saturation}

By the \emph{language of dynamical systems} we mean the formal logical language that extends the language of Boolean algebras by adding a function symbol representing an automorphism of a Boolean algebra. 
A model for this language is a dynamical system $\< \mathbb A,\alpha \>$, where $\mathbb A$ is a Boolean algebra and $\alpha: \mathbb A \to \mathbb A$ an automorphism. Things like $+$ or $\leq$ are understood as part of the language of the Boolean algebra~$\mathbb A$, and not listed explicitly.

\begin{definition} 
If $\alpha$ is an automorphism of $\pnmf$, let 
\[
\fA_\alpha \,=\, \< \pnmf,\alpha \>. 
\]
let $\mathrm{Th}(\fA_\alpha)$ denote the first-order theory of this structure in the language of dynamical systems, and let $\mathrm{Th}_\exists(\fA_\alpha)$ denote its existential theory. 
For convenience, 
if~$f$ is an almost permutation of $\N$, we write $\fA_f$ rather than $\fA_{\alpha_f}$. 
Similarly, if $\alpha_{\bar n}$ is a rotary automorphism, we write $\fA_{\bar n}$ rather than $\fA_{\alpha_{\bar n}}$.
\hfill\Coffeecup
\end{definition}
As usual, $\mathrm{Th}(\fA_f) = \mathrm{Th}(\fA_g)$ can also be expressed by saying $\fA_f$ and $\fA_g$ are \emph{elementarily equivalent}, denoted $\fA_f \equiv \fA_g$. 
Furthermore, observe that the following important fact follows immediately from the definitions:

\begin{proposition}\label{prop:Ref}
Given $\phi,\psi \in \Aut(\pnmf)$, the structures $\fA_\phi$ and $\fA_\psi$ are isomorphic if and only if the maps $\phi$ and $\psi$ are conjugate.
\end{proposition}

The main result of this section, Theorem~\ref{thm:ExtensionToTriv} below, states: 
assuming \ch, if $\parity(f) = \parity(g)$ and $\fA_f \equiv \fA_g$ then $\fA_f \cong \fA_g$. 
A special case of this is the main result of  \cite{brian2024does}, which states that \ch implies $\fA_\s \cong \fA_{\s^{-1}}$. 
We do not claim to give an alternative proof of the results in \cite{brian2024does}; rather, our theorem extends the main theorem of \cite{brian2024does}, and its proof relies directly on \cite{brian2024does}. The proof of Theorem~\ref{thm:ExtensionToTriv} uses \cite{brian2024does} as a black box, so that the reader does not need to be familiar with the details in \cite{brian2024does} in order to understand the proof.

%As a corollary to Theorem~\ref{thm:ExtensionToTriv}, we also show (as promised in the introduction) that $\fA_f \cong \fA_g$ in a forcing extension if and only if $\parity(f) = \parity(g)$ and $\fA_f \equiv \fA_g$.

The proof of Theorem~\ref{thm:ExtensionToTriv} relies on a saturation argument. Since some of our readership may not care about $\aleph_1$ and since we will need $\kappa$-saturated models for $\kappa>\aleph_1$  only once, we adopt the following variation on the standard terminology (already standard in the theory of \cstar-algebras). 

\begin{definition} We say that a first-order structure $\fA$ is \emph{countably saturated} if every consistent type over a countable subset of $\fA$ is realized in $\fA$.
\end{definition} 

We  show that rotary automorphisms correspond to countably saturated structures (Lemma~\ref{L.ReducedProduct}). To complement this argument, we also classify precisely which trivial maps give rise to countably saturated structures: 
the structure $\fA_f$ is countably saturated if and only if $\mathrm{Z}(f) \neq s_{\Z \times \Z}$  (the permutation with infinitely many bi-infinite orbits, Definition~\ref{Def.sZZ})
This is Theorem~\ref{thm:Saturation} below, whose proof also relies on \cite{brian2024does}. Unlike  the proof of Theorem~\ref{thm:ExtensionToTriv}, the proof of Theorem~\ref{thm:Saturation} requires the reader to be familiar with some of the details of \cite{brian2024does}.

We begin by showing that if $\alpha_{\bar n}$ is a (trivial) rotary automorphism of $\pnmf$, then the structure $\fA_{\bar n}$ is countably saturated. 
For the statement of the following lemma, let $\P(n)$ denote the power set of $n = \{0,1,2,\dots,n-1\}$, which is a finite Boolean algebra of size $2^n$, and let $r_n$ denote the automorphism induced on $\P(n)$ by the function $i \mapsto i+1 \ \, (\text{mod } n)$. %(That is, if $A \sub n$ then $r_n(A) = A+1$, computed modulo $n$.)

\begin{lemma}\label{L.ReducedProduct}
	For every sequence $\bar n = \seq{n_j}{n \in \w}$ of natural numbers,
	\[
	\fA_{\bar n} \,\cong\, \textstyle \Big( \prod_{j \in \w} \< \P(n_j),r_{n_j} \> \Big) /\Fin. 
	\]
	Consequently, $\fA_{\bar n}$ is countably saturated. 
\end{lemma}

\begin{proof}
As in Definition~\ref{def:RotaryMap}, 
let $I_0,I_1,I_2,\dots$ be the sequence of adjacent intervals with $|I_j| = n_j$ for all $j$, and let $r_{\bar n}$ denote the map 
$$r_{\bar n}(i) \,=\, \begin{cases}
i+1 &\text{ if } i \neq \max I_k \text{ for any $k$,} \\
\min I_k &\text{ if } i = \max I_k.
\end{cases}$$
For each $j$, this map induces a function $A \mapsto r_{\bar n}[A]$ on $\P(I_{n_j})$; let us denote this function $R_j$. Clearly $\< \P(I_j),R_j \> \cong \< \P(n),r_n \>$. 

Define a function $f \colon \P(\N) \to \prod_{j \in \w} \P(I_j)$ by setting 
$$f(X) \,=\, \seq{ X \cap I_j }{j \in \w}.$$
If $[X]_\mathrm{Fin} = [Y]_\mathrm{Fin}$ then $X \cap I_j = Y \cap I_j$ for all but finitely many $j$. 
Thus $f$ lifts to an isomorphism 
mapping $\< \pnmf,\alpha_{\bar n} \>$ to $\prod_j \< \P(I_j),R_j \> /\Fin$, namely the function $[X]_\mathrm{Fin} \,\mapsto\, [f(X)]_\mathrm{Fin}$. 
The first assertion of the lemma follows. 

Every reduced product over $\Fin$ is countably saturated by \cite{JonOl:Almost} (a proof of this result in a more general context can be found in \cite[Theorem 16.5.1]{Fa:STCstar},
and an elegant direct proof can be found in \cite[Theorem~4]{palmgren}), so the second assertion of the lemma follows from the first. 
\end{proof}

\begin{lemma}\label{lem:simC}
Let $\alpha_{\bar m}$ and $\alpha_{\bar n}$ be rotary automorphisms of $\pnmf$. 
Assuming \ch,  $\alpha_{\bar m}$ and $\alpha_{\bar n}$ are conjugate if and only if $\fA_{\bar m} \equiv \fA_{\bar n}$. 
\end{lemma}
\begin{proof}%[Proof of Proposition~\ref{prop:simC}]
Recall (Proposition~\ref{prop:Ref}) $\fA_{\bar m}$ and $\fA_{\bar n}$ are isomorphic if and only if $\alpha_{\bar m}$ and $\alpha_{\bar n}$ are conjugate. 
With this in mind, the forward implication of the lemma is trivial: isomorphic structures have the same theory. 
For the reverse implication, assume \ch, which means $|\pnmf| = \aleph_1$. 
Any two elementarily equivalent countably saturated structures of size $\aleph_1$ are isomorphic (see, e.g., \cite[Theorem 5.1.13]{ChaKe},  keeping in mind that our `countably saturated' translates as `$\aleph_1$-saturated'). 
Thus, applying Lemma~\ref{L.ReducedProduct}, if $\fA_{\bar m} \equiv \fA_{\bar n}$ then $\fA_{\bar m} \cong \fA_{\bar n}$, i.e. $\alpha_{\bar m}$ and $\alpha_{\bar n}$ are conjugate. 
\end{proof}

Proposition~\ref{P.FVtrick} below is based on \cite[Lemma 5.2]{ghasemi2016reduced}. 	An analogous fact is true for metric structures, with an essentially identical proof.

\begin{proposition}[Ghasemi's Trick] \label{P.FVtrick} Suppose that $A_n$, for $n\in \bbN$, is a sequence of structures in the same countable language $\calL$. There is a sequence $n(k)$, for $k\in \bbN$, such that for every further subsequence $n(k(i))$, for $i\in \bbN$, the reduced products $\prod_k A_{n(k)}/\Fin$ and $\prod_i A_{n(k(i))}/\Fin$ are elementarily equivalent and therefore isomorphic, provided that \ch{} holds and each structure has cardinality $\leq\! 2^{\aleph_0}$. 
\end{proposition}

\begin{proof} 
Since the language of dynamical systems is countable, there are only countably many $\calL$-sentences.  Enumerate these sentences as $\varphi_j$, for $j\in \bbN$. 
Using this fact, a standard diagonalization argument shows that there is an infinite set $D = \{ n(k) \mid k \in \bbN \}$ such that for every~$j$, $\varphi_j^{A_{n(k)}}$ is either true for all but finitely many $k$ or it is false for all but finitely many $k$. (This is essentially identical to the usual argument that the splitting number $\mathfrak s$ is uncountable.)

To see that this subsequence is as required, choose a further subsequence $n(k(i))$ of this sequence, and fix an $\calL$-sentence $\varphi$. 
By the Feferman--Vaught theorem (\cite{feferman1959first}, also \cite[Proposition~6.3.2]{ChaKe}) there are a finite set $F_\varphi=\{\psi_i\mid i<m\}$ of $\calL$-sentences and a formula $\Theta(Z_0,\dots, Z_{m-1})$ in the language of Boolean algebras such that,  with $Z(\psi)$ denoting the $\Fin$-equivalence class of the set $\{n \mid B_n\models \psi\}$,  we have that $\big( \prod_n B_n \big) / \Fin \models \varphi$ if and only if $\pnmf \models \Theta(Z(\psi_0), \dots, Z(\psi_{m-1}))$.  
		
		If $\psi$ is a Boolean combination of the $\psi_j$, for $j<m$, then it  is either true in all but finitely many of the $A_{n(k)}$ or it is false in all but finitely many of the $A_{n(k)}$, and the same is true for every further subsequence.  
		Therefore $\big( \prod_k A_{n(k)} \big)/\Fin\models\varphi$ if and only if $\big( \prod_i A_{n(k(i))} \big)/\Fin\models\varphi$. Since $\varphi$ was arbitrary, this proves that the reduced products are elementarily equivalent. 
		By the countable saturation of reduced products over $\Fin$ (already used in the proof of Lemma~\ref{L.ReducedProduct}), these structures are countably saturated. Since they are  also elementarily equivalent, as in the proof of Lemma~\ref{L.ReducedProduct}, \ch{}  implies they are isomorphic. 
	\end{proof}

\begin{lemma}\label{L.subsequence}
Suppose $\bar n$ is a sequence in $\N$, and $f_{\bar n}$ the corresponding trivial automorphism of $\N$. Then $d \in \pnmf$ is $\alpha_{\bar n}$-invariant if and only if $d=[D]_\mathrm{Fin}$ for some $D \sub \N$ that is a union of $f_{\bar n}$-orbits. 
Furthermore, if $d$ is a nonzero $\alpha_{\bar n}$-invariant set, then the restriction of $\alpha_{\bar n}$ to $d$ is conjugate to $\alpha_{\bar m}$ for some subsequence $\bar m$ of $\bar n$. 
\end{lemma}

\begin{proof}
If $D \sub \N$ is a union of $f_{\bar n}$-orbits, then $f_{\bar n}[D] = D$, and therefore $\alpha_{\bar n}([D]_\mathrm{Fin}) = [f_{\bar n}[D]]_{\mathrm{Fin}} = [D]_{\mathrm{Fin}}$, i.e., $[D]_{\mathrm{Fin}}$ is $\alpha_{\bar n}$-invariant. 
Conversely, suppose $[D]_\mathrm{Fin}$ is not $\alpha_{\bar n}$-invariant. Then either $\alpha_{\bar n}([D]_\mathrm{Fin}) - [D]_\mathrm{Fin} \neq [\emptyset]$ or $[D]_\mathrm{Fin} - \alpha_{\bar n}([D]_\mathrm{Fin}) \neq [\emptyset]$. In the former case, 
$$\alpha_{\bar n}([D]_\mathrm{Fin}) - [D]_\mathrm{Fin} \,=\, [f_{\bar n}[D]]_\mathrm{Fin} - [D]_\mathrm{Fin} \,=\, [f_{\bar n}[D] \setminus D]_\mathrm{Fin} \,\neq\, [\emptyset],$$
which means $f_{\bar n}[D] \setminus D$ is infinite. In particular $D$ is not $f_{\bar n}$-invariant, and therefore is not a union of $f_{\bar n}$-orbits. Similarly, in the latter case $[D]_\mathrm{Fin} - \alpha_{\bar n}([D]_\mathrm{Fin}) \neq [\emptyset]$ implies $D \setminus f_{\bar n}[D]$ is infinite, which means $D$ is not $f_{\bar n}$-invariant, and therefore is not a union of $f_{\bar n}$-orbits. 

For the ``furthermore'' part of the lemma, fix a nonzero $d \in \pnmf$ such that $\alpha_{\bar n}(d) = d$. 
By the previous paragraph, there is some $D \sub \N$ with $d = [D]_\mathrm{Fin}$ such that $D$ is a union of $f_{\bar n}$-orbits. Letting $\bar m$ denote the subsequence of $\bar n$ consisting of those orbits contained in $D$, it is clear that $f_{\bar m}$ is equivalent to $f_{\bar n} \restriction D$, which implies that $\alpha_{\bar m}$ is conjugate to $\alpha_{\bar n} \restriction d$.
\end{proof}

The following consequence of Proposition~\ref{P.FVtrick} will be used in \S\ref{S.uRc}.

\begin{corollary}\label{C.many.nonconjugate} There is a family $\mathcal G$ of $2^{\aleph_0}$ trivial automorphisms of $\pnmf$ with the following properties: 
	\begin{enumerate}
\item \label{1.many.conjugate} Assuming \ch{}: if $\gamma,\gamma’ \in \mathcal G$, the restriction of  ${\gamma}$ to any nonzero ${\gamma}$-invariant set is conjugate to the restriction of ${\gamma'}$ to any nonzero ${\gamma'}$-invariant set. 
\item  \label{2.many.conjugate} Assuming all automorphisms of $\pnmf$ are trivial (e.g., assuming $\mathsf{OCA}_{\mathrm{T}}$,  by \cite[Theorem~1]{de2023trivial}): if $\gamma,\gamma' \in \mathcal G$ and $\gamma \neq \gamma'$, then the restriction of ${\gamma}$ to any nonzero ${\gamma}$-invariant set is not conjugate to the restriction of ${\gamma'}$ to any nonzero ${\gamma'}$-invariant set. 
	\end{enumerate}
\end{corollary}

\begin{proof} 
Let $\bar n = \seq{n_j}{j \in \bbN}$ denote any increasing sequence in $\N$ (e.g., take $n_j = j$ for concreteness). 
For each $j$, let $A_j = \< \mathcal P(n_j),r_{n_j} \>$, where, just as in Lemma~\ref{L.ReducedProduct}, $r_{n_j}$ denotes the automorphism induced on $\mathcal P(n_j)$ by the map $i \mapsto i+1$ (mod $n_j$). 
Each $A_j$ is a finite structure in the language of dynamical systems. 

Applying Proposition~\ref{P.FVtrick}, there is a subsequence $\bar m = \seq{m_j}{j \in \w}$ of $\bar n$ such that if $\seq{m_{k(i)}}{i \in \N}$ and $\seq{m_{k'(i)}}{i \in \N}$ are any two subsequences of $\bar m$ (i.e., $k,k'$ are increasing functions $\N \to \N$), then the reduced products $\big( \prod_{i \in \N}A_{m_{k(i)}} \big)/\mathrm{Fin}$ and $\big( \prod_{i \in \N}A_{m_{k'(i)}} \big)/\mathrm{Fin}$ are elementarily equivalent. 
Via Lemma~\ref{L.ReducedProduct}, this means that if $\bar a$ and $\bar b$ are any subsequences of $\bar m$, then $\fA_{\bar a} \equiv \fA_{\bar b}$.

Let $\mathcal A$ be a family of infinite subsets of $\N$ such that any two members of $\mathcal A$ are almost disjoint (i.e., disjoint modulo $\Fin$) and such that $|\mathcal A| = 2^{\aleph_0}$ (see, e.g., \cite[Theorem~9.2.2]{Fa:STCstar}). 
Enumerating each of the members of $\mathcal A$ in increasing order, we obtain a family $\mathcal F$ of $2^{\aleph_0}$ strictly increasing functions $\N \to \N$, such that the ranges of any two of these functions have only finite overlap. 
For each $f \in \mathcal F$, let $\bar m_f = \seq{m(f(i))}{i \in \N}$ denote the corresponding subsequence of $\bar m$. 

Let $\mathcal G = \{ \alpha_{\bar m_f} \mid f \in \mathcal F \}$. 
Fix $f,f' \in \mathcal F$, let $d$ be a nonzero $\alpha_{\bar m_f}$-invariant set, and let $d'$ be a nonzero $\alpha_{\bar m_{f'}}$-invariant set. 
By Lemma~\ref{L.subsequence}, there is a subsequence $\bar k$ of $\bar m_f$ such that $\alpha_{\bar k}$ is conjugate to the restriction of $\alpha_{\bar m_f}$ to $d$, and there is a subsequence $\bar k'$ of $\bar m_{f'}$ such that $\alpha_{\bar k'}$ is conjugate to the restriction of $\alpha_{\bar m_{f'}}$ to $d'$. 

To prove~\eqref{1.many.conjugate}, observe that $\bar k$ and $\bar k'$ are both subsequences of $\bar m$, and therefore $\fA_{\bar k} \equiv \fA_{\bar k'}$. 
Furthermore, these two structures are countably saturated by Lemma~\ref{L.ReducedProduct}. 
Consequently, as in the proof of Lemma~\ref{lem:simC}, \ch implies that  $\fA_{\bar k}$ and $\fA_{\bar k'}$ are isomorphic, which means $\alpha_{\bar k}$ and $\alpha_{\bar k'}$ are conjugate by Proposition~\ref{prop:Ref}. 
Thus \ch implies the restriction of $\alpha_{\bar m_f}$ to $d$ is conjugate to the restriction of $\alpha_{\bar m_{f'}}$ to $d'$.

To prove~\eqref{2.many.conjugate}, suppose $f$ and $f'$ are distinct members of $\cF$.  
By our choice of the family $\mathcal F$, there are only finitely many numbers appearing in both the sequences $\bar k$ and $\bar k'$. 
However, it is clear in general that if two rotary permutations of $\N$ are equivalent (conjugate via an almost permutation of $\N$), then they have the same cycle structure modulo $\mathrm{Fin}$ (i.e., if $r_{\bar \ell} \sim r_{\bar \ell'}$ then the image of $\bar \ell$ is equal to the image of $\bar \ell'$ modulo finite). 
Hence $r_{\bar k}$ and $r_{\bar k'}$ are not equivalent. 
By Lemma~\ref{lem:TriviallyConjugate}, this means $\alpha_{\bar k}$ and $\alpha_{\bar k'}$ are not trivially conjugate. Assuming all automorphisms of $\pnmf$ are trivial, this means $\alpha_{\bar k}$ and $\alpha_{\bar k'}$ are not conjugate, and thus the restriction of $\alpha_{\bar m_f}$ to $d$ is not conjugate to the restriction of $\alpha_{\bar m_{f'}}$ to $d'$.
\end{proof} 

The following uses the notation introduced in Definition~\ref{def:OrbitStructure}  and Lemma~\ref{lem:OrbitStructure}. 

\begin{lemma}\label{lem:Rotary}
Suppose that $f$ and $g$ are almost permutations of $\N$ and suppose that $f = \mathrm R(f) \oplus \mathrm S(f) \oplus \mathrm Z(f)$ and $g = \mathrm R(g) \oplus \mathrm S(g) \oplus \mathrm Z(g)$. 
Assuming \ch, if $\fA_f \equiv \fA_g$ then $\alpha_{\mathrm R(f)}$ and $\alpha_{\mathrm R(g)}$ are conjugate. 
\end{lemma}
\begin{proof}%[Proof of Proposition~\ref{prop:simC}]
It follows from Lemma~\ref{lem:DefiningTheRotary} that $\alpha_{\mathrm R(f)}$ is definable in $\fA_f$, as the restriction of $\alpha_f$ to the complement of the supremum of the components of $\alpha_f$. 
But of course $\alpha_{\mathrm R(g)}$ is definable in $\fA_g$, with the same definition, so $\fA_f \equiv \fA_g$ implies $\fA_{\mathrm R(f)} \equiv \fA_{\mathrm R(g)}$. 
By Lemma~\ref{lem:simC}, \ch implies $\fA_{\mathrm R(f)} \cong \fA_{\mathrm R(g)}$, i.e., $\alpha_{\mathrm R(f)}$ and $\alpha_{\mathrm R(g)}$ are conjugate. 
\end{proof}

\begin{theorem}\label{thm:ExtensionToTriv}
Let $\alpha$ and $\beta$ be trivial automorphisms of $\pnmf$. Assuming \ch, $\alpha$ and $\beta$ are conjugate if and only if $\parity(\alpha) = \parity(\beta)$ and $\fA_\alpha \equiv \fA_\beta$. 
\end{theorem}
\begin{proof}
The forward implication is a trivial consequence of Proposition~\ref{prop:Ref}: if $\alpha$ and $\beta$ are conjugate automorphisms of $\pnmf$, then $\parity(\alpha) = \parity(\beta)$ by Theorem~\ref{thm:Fredholm}; also, $\alpha$ and $\beta$ being conjugate is equivalent to $\fA_\alpha \cong \fA_\beta$, which implies $\fA_\alpha \equiv \fA_\beta$. 

For the reverse implication, let $f$ and $g$ be almost permutations of $\N$ with $\alpha = \alpha_f$ and $\beta = \alpha_g$, and suppose that $\parity(\alpha_f) = \parity(\alpha_g)$ and $\fA_{\alpha_f} \equiv \fA_{\alpha_g}$. 

By Lemma~\ref{lem:TriviallyConjugate}, we may replace $f$ or $g$ with any equivalent almost permutation without changing the conjugacy class of $\alpha_f$ or $\alpha_g$. 
Using this fact and Lemma~\ref{lem:OrbitStructure}, we may assume $f = \mathrm{R}(f) \oplus \mathrm{Z}(f) \oplus \mathrm{S}(f)$ and $g = \mathrm{R}(g) \oplus \mathrm{Z}(g) \oplus \mathrm{S}(g)$. 

The main theorem of \cite{brian2024does} states that \ch implies $\s$ and $\s^{-1}$ are conjugate. It follows that, without changing the conjugacy class of $\alpha_f$, we may replace an $\N$-like orbit of $f$ with a reverse $\N$-like orbit, or we may replace a pair of $\N$-like orbits or reverse $\N$-like orbits with a $\Z$-like orbit. (The resulting map $f'$ will have the property that $\alpha_f$ and $\alpha_{f'}$ are conjugate; let us note, however, that $\alpha_f$ and $\alpha_{f'}$ are not trivially conjugate.) 
Applying this process repeatedly (but finitely), we may (and do) assume that $\mathrm{S}(f)$ and $\mathrm{S}(g)$ are each either equal to $s$ or the empty map. 

Because $\parity(f) = \parity(g)$, we cannot have $\mathrm{S}(f) = s$ (which implies $\parity(f)=1$) while $\mathrm{S}(g)=\emptyset$ (which implies $\parity(g)=0$), or vice versa. Thus, by the previous paragraph, either $\mathrm{S}(f) = \mathrm{S}(g) = s$, or else $\mathrm{S}(f)$ and $\mathrm{S}(g)$ and both empty. 

Recall that components of $\fA_\alpha$ are the atoms of $\FIX_\alpha$ (Definition~\ref{def:Components}). 
In a dynamical system $\fA_\alpha$, for $c\in \cPNF$, 
the assertion ``$c$ is a component of $\alpha$'' is expressible as a first-order sentence:
$$(\alpha(c) = c)  \wedge ([\emptyset] < x < c \ \Rightarrow \ \alpha(c) \neq x).$$
It follows that for any $n \geq 0$, the assertion ``$\alpha$ has precisely $n$ components'' is expressible as a first-order sentence. Because $\fA_f \equiv \fA_g$, this means $\alpha_f$ and $\alpha_g$ either both have infinitely many components, or if not, then both have the same finite number of components. 
%We consider these two cases separately.

Combining the previous two paragraphs with Lemma~\ref{lem:FIX3}, we get $\mathrm{Z}(f) = \mathrm{Z}(g)$ and $\mathrm{S}(f) = \mathrm{S}(g)$. 
Furthermore, $\alpha_{\mathrm R(f)}$ and $\alpha_{\mathrm R(g)}$ are conjugate by Lemma~\ref{lem:Rotary}. 
It follows that $\fA_{\alpha_f} \cong \fA_{\alpha_g}$.
\end{proof}

The second-to-last paragraph of the proof of Theorem~\ref{thm:ExtensionToTriv} shows that if $\fA_f \equiv \fA_g$ and $\mathrm{Z}(f) \neq s_{\Z \times \Z}$, then $\mathrm{Z}(g) \neq s_{\Z \times \Z}$, and $\alpha_f$ and $\alpha_g$ both have the same finite number $n$ of components. Furthermore, the earlier part of the proof shows that when this is the case, then $\parity(\alpha_f) = \parity(\alpha_g) = \mathrm{parity}(n)$. This shows:

\begin{proposition}\label{P.parity}
Let $\alpha_f$ and $\alpha_g$ be trivial automorphisms of $\pnmf$, and suppose $\mathrm{Z}(f),\mathrm{Z}(g) \neq s_{\Z \times \Z}$. Assuming \ch, $\alpha_f$ and $\alpha_g$ are conjugate if and only if $\fA_\alpha \equiv \fA_\beta$. 
\end{proposition}

 In other words, the parity obstruction from Theorem~\ref{thm:Fredholm} only comes into play when $\mathrm{Z}(f)=\mathrm{Z}(g) = s_{\Z \times \Z}$, because otherwise $\fA_f \equiv \fA_g$ implies $\parity(\alpha_f) = \parity(\alpha_g)$ automatically. This naturally raises a question (see Question~\ref{Q.parity}).

\subsection{Potential conjugacy}
Given an almost permutation $f$ of $\N$, the definition of $\alpha_f$ is absolute between different models of set theory, in the following sense. 
If $V$ and $W$ are two different models of set theory with $f \in V \cap W$, and if $A \subseteq \N$ with $A \in V \cap W$, then $\alpha_f([A]_\mathrm{Fin}) = [f[A]]_\mathrm{Fin}$ is the same thing in $V$ and in $W$. Only the domain and range of $\alpha_f$ can differ between $V$ and $W$, and they do whenever $\mathcal P(\N)^V \neq \mathcal P(\N)^W$. Because of this absoluteness, we suppress the superscripts in  $\alpha_f^V$ or $\alpha_f^{V[G]}$ in the following definition and corollary.

\begin{definition} \label{Def.Potentially} Two trivial automorphisms $\alpha_f$ and $\alpha_g$ of $\pnmf$ are \emph{potentially conjugate} if there is a set forcing extension in which $\alpha_f$ and $\alpha_g$ are conjugate. 
\hfill\Coffeecup
\end{definition}

\begin{corollary}\label{C.Potentially}
	For two trivial automorphisms $\alpha$ and $\beta$ of $\pnmf$ the following are equivalent. 
\begin{enumerate}
	\item  $\alpha$ and $\beta$ have the same index parity and $\fA_\alpha \equiv \fA_\beta$. 
	\item  $\alpha$ and $\beta$ are potentially conjugate. 
\end{enumerate}	
\end{corollary}

\begin{proof}
Suppose $\alpha$ and $\beta$ are trivial automorphisms of $\pnmf$ with $\parity(\alpha) = \parity(\beta)$ and $\fA_\alpha \equiv \fA_\beta$. 
Force with countable conditions to collapse the continuum $\mathfrak c$ to $\aleph_1$. 
This forcing adds no new subsets of $\N$ (see e.g., \cite[Theorem~8.3]{Ku:Book}), which means that $\pnmf$ is exactly the same in $V$ and in $V[G]$. 
Because of this, each of the structures $\< \pnmf,\alpha \>$ and $\<\pnmf,\beta \>$ is the same in $V$ and in $V[G]$. 
By the absoluteness of the satisfaction relation, the first-order theories of these structures is likewise unchanged. 
Consequently, $\fA_\alpha \equiv \fA_\beta$ in $V[G]$. 
Similarly, $\parity(\alpha)$ and $\parity(\beta)$ are unchanged between $V$ and $V[G]$, because parity is defined by counting the elements in certain finite sets and subtracting. 
Thus $\parity(\alpha) = \parity(\beta)$ and $\fA_\alpha \equiv \fA_\beta$ in $V[G]$. By Theorem~\ref{thm:ExtensionToTriv}, $\fA_\alpha \cong \fA_\beta$ in $V[G]$. Hence $\alpha$ and $\beta$ are potentially conjugate.

For the converse direction, let $\alpha$ and $\beta$ be potentially conjugate trivial automorphisms of $\pnmf$: i.e., there is some forcing extension $V[G]$ in which $\fA_\alpha \cong \fA_\beta$. This implies $\fA_\alpha \equiv \fA_\beta$ in $V[G]$, and by Theorem~\ref{thm:Fredholm}, it also implies $\parity(\alpha) = \parity(\beta)$ in $V[G]$. 
Fix almost permutations $f,g$ of $\N$ with $\alpha = \alpha_f$ and $\beta = \alpha_g$. 

The parity of $f$ and $g$, defined as the difference between two finite sets, is absolute. 
Since $\parity(f) = \parity(\alpha_f) = \parity(\alpha_g) = \parity(g)$ in $V[G]$, we have $\parity(\alpha_f) = \parity(f) = \parity(g) = \parity(\alpha_g)$ in $V$. 

To finish the proof, we must show $\fA_\alpha \equiv \fA_\beta$ in $V$. 
The definitions of $\mathrm{R}(f)$, $\mathrm{Z}(f)$, and $\mathrm{S}(f)$, and the corresponding pieces of $g$, are absolute. 
Therefore it suffices to show $\fA_{\mathrm{R}(f)} \equiv \fA_{\mathrm{R}(g)}$ and $\fA_{\mathrm{Z}(f) \oplus \mathrm{S}(f)} \equiv \fA_{\mathrm{Z}(g) \oplus \mathrm{S}(g)}$. 

For the rotary functions $\mathrm{R}(f)$ and $\mathrm{R}(g)$, note that $\fA_{\mathrm{R}(f)}$ and $\fA_{\mathrm{R}(g)}$ are expressible as reduced products, by Lemma~\ref{L.ReducedProduct}. Therefore the  theories of these structures are effectively determined by the sequence of the theories of the finite dynamical systems in the product (and are even decidable) by the Feferman--Vaught theorem (\cite{feferman1959first}, \cite[Proposition~6.3.2]{ChaKe}). These sequences of finite structures are determined in an effective (hence absolute) way from $f$ and $g$. 
Thus the theories of $\fA_{\mathrm{R}(f)}$ and $\fA_{\mathrm{R}(g)}$ are absolute. As $\fA_{\mathrm{R}(f)} \equiv \fA_{\mathrm{R}(g)}$ in $V[G]$, this implies $\fA_{\mathrm{R}(f)} \equiv \fA_{\mathrm{R}(g)}$ in the ground model. 

For ${\mathrm{Z}(f) \oplus \mathrm{S}(f)}$ and ${\mathrm{Z}(g) \oplus \mathrm{S}(g)}$, we consider two cases. 

If $\mathrm{Z}(f) \neq s_{\Z \times \Z}$, then $\alpha_f$ has finitely many components. As in the proof of Theorem~\ref{thm:ExtensionToTriv}, the statement ``$\alpha$ has precisely $n$ components" is expressible as a first-order sentence. 
Because $\fA_{\mathrm{Z}(f) \oplus \mathrm{S}(f)} \equiv \fA_{\mathrm{Z}(g) \oplus \mathrm{S}(g)}$ in $V[G]$, this means $\alpha_f$ and $\alpha_g$ have the same finite number of components in $V[G]$. 
But this finite number is determined in a simple and absolute way from $f$ and $g$, and thus $\alpha_f$ and $\alpha_g$ have the same finite number of components in $V$, say $n$. 
In this case $\alpha_f$ and $\alpha_g$ are both (up to equivalence) simply direct sums of $n$ copies of $\sigma$ and $\sigma^{-1}$. But $\fA_\sigma \equiv \fA_{\sigma^{-1}}$ by \cite{brian2024does}, so in this case $\fA_{\mathrm{Z}(f) \oplus \mathrm{S}(f)} \equiv \fA_{\mathrm{Z}(g) \oplus \mathrm{S}(g)}$. 

If $\mathrm{Z}(f) = s_{\Z \times \Z}$, then $\alpha_f$ has infinitely many components. 
By the previous paragraph, if $\alpha_f$ has finitely many components then $\alpha_g$ does too, and by the same logic, if $\alpha_g$ has finitely many components then $\alpha_f$ does too. 
Therefore, in this case, $\alpha_g$ has infinitely many components as well. 
Thus $\mathrm{Z}(f) = s_{\Z \times \Z} = \mathrm{Z}(g)$. 
Now recall $\parity(f) = \parity(g)$, which means $\mathrm{S}(f)$ and $\mathrm{S}(g)$ either both contain an odd finite number of orbits, or both contain an even finite number of orbits. 
%In either case, let us fix a number $N$ greater than the number of orbits in $\mathrm{S}(f)$ or $\mathrm{S}(g)$, and differing from either by an even number. 
Recall that a single $\Z$-like orbit is trivially equivalent to an $\N$-like orbit and a reverse $\N$-like orbit. Thus, by decomposing some of the (infinitely many) orbits in $\mathrm{Z}(f) = s_{\Z \times \Z}$ or in $\mathrm{Z}(g) = s_{\Z \times \Z}$, we may (and do) assume $\mathrm{S}(f)$ and $\mathrm{S}(g)$ both contain the same number of orbits. 
But $\fA_\sigma \equiv \fA_{\sigma^{-1}}$ by \cite{brian2024does}, so this means $\fA_{\mathrm{S}(f)} \equiv \fA_{\mathrm{S}(g)}$. Because $\mathrm{Z}(f) = s_{\Z \times \Z} = \mathrm{Z}(g)$, we conclude that $\fA_{\mathrm{Z}(f) \oplus \mathrm{S}(f)} \equiv \fA_{\mathrm{Z}(g) \oplus \mathrm{S}(g)}$ as claimed.
\end{proof}

Combining this corollary with Theorem~\ref{thm:ExtensionToTriv} completes the proof of Theorem A from the introduction. 

\vspace{2mm}

\subsection{On the failure of countable saturation in $\fA_\alpha$}
Lemma~\ref{lem:simC} might cause one to hope that~$\fA_\alpha$ is countably saturated for every trivial automorphism $\alpha$. This fleeting hope is dispelled in the following proposition.

\begin{proposition}\label{prop:Saturation}
If $f = s_{\Z \times \Z}$ then then $\fA_{\alpha_f}$ is not countably saturated. 
\end{proposition}
\begin{proof}  
For convenience, let us take $\mathrm{dom}(f) = \Z \times \Z$, with $f(m,n) = (m+1,n)$. 
Let $A = \N \times \Z$. The Boolean algebra
$$\mathbb A \,=\, \set{ c \wedge [A]_{\mathrm{Fin}} }{c \in \FIX_{\alpha_f}}$$
is a definable substructure of $\fA_f$. (We do not care about how the map $\alpha_f$ acts on $\mathbb A$, so we interpret this structure in the reduced language of Boolean algebras.) 
Applying Lemma~\ref{lem:FIX}, $\mathbb A \cong \cP(\bbN)$. But $\mathcal P(\N)$ is not a countably saturated Boolean algebra. Because $\fA_f$ has a definable substructure that is not countably saturated (in the reduced language), $\fA_f$ itself is not countably saturated. 
\end{proof}

This proposition, together with Lemma~\ref{L.ReducedProduct}, raises the question of exactly when~$\fA_f$ is countably saturated. 
Towards answering this question, we turn next to the shift map $\s$. 
The following is a consequence of the main result of \cite{brian2024does}. 

\begin{proposition}\label{prop:ShiftSaturation}
$\fA_\s$ and $\fA_{\s^{-1}}$ are countably saturated, and their theories are model complete. 
\end{proposition}
\begin{proof}  
For this proof (and this proof only), we assume familiarity with the terminology and notation in \cite{brian2024does}. 
The proof is no different for $\s$ and for $\s^{-1}$, so we give the argument just for $\s$. 
A structure is both countably universal and countably homogeneous if and only if it is countably saturated (\cite[Theorem 5.1.14]{ChaKe}). This is how we shall prove $\fA_\s$ is countably saturated. 
To begin, we need the following strengthening of Lemma 6.2 (the Lifting Lemma) from \cite{brian2024does}:

%\vspace{1.5mm}

 \begin{claim} Let $(\< \mathbb A,\alpha \>,\< \mathbb B,\beta \>,\iota,\eta)$ be an instance of the lifting problem for $\< \pnmf,\s \>$. If $\< \mathbb A,\alpha \>$ is elementarily equivalent to $\< \pnmf,\s \>$, then this instance of the lifting problem has a solution. 
\end{claim}

\begin{proof}[Proof of the Claim] The key observation is that the proof of Theorem 10.2 in \cite{brian2024does} (all theorem numbers in this proof refer to \cite{brian2024does}) can be improved by putting an explicit upper bound on the size of the digraph $\< \mathcal C,\toC \>$. 
Specifically, the construction in the main part of the proof of Theorem 10.2 shows 
$$|\mathcal C| \,\leq\, 1+ 2^{2^{2 \cdot |\mathcal B|^2}}.$$
This bound is not meant to be tight, and comes from just following through the stages of the construction of $\mathcal C$ in the proof of Theorem 10.2, and computing the trivial bounds at every step. 
Specifically, there are $2^{|\mathcal B|^2}$ possibilities for each $\mathcal P^i$ and $\mathcal F^i$, simply because these are binary relations on $\mathcal B$. Hence there are $\big( 2^{|\mathcal B|^2} \big)^2 = 2^{2 \cdot |\mathcal B|^2}$ possibilities for $(\mathcal P^i,\mathcal F^i)$, and thus $2^{2^{2 \cdot |\mathcal B|^2}}$ possibilities for the $\mathcal L^k$. Because $\mathcal C$ is formed from a subset of the $\mathcal L^k$, plus one extra vertex, the bound follows. 

The observant reader of \cite{brian2024does} may notice that Theorem 10.8 forms part of the proof of Theorem 10.2, and also contains a construction of a digraph $\< \mathcal C,\toC \>$ to cover a particular case of Theorem 10.2. This does not ruin the above bound, though, because the digraph constructed in the proof of Theorem 10.8 is even smaller. Going through the construction in the proof of Theorem 10.8, and taking easy bounds at every step, shows $|\mathcal C| \leq |\mathcal A| + 2^{|\mathcal B|} \leq |\mathcal B| + 2^{|\mathcal B|}$.

With this improvement to Theorem 10.2 in mind, one can make a similar change to Theorem 11.2: not only is $\< \pnmf,\sigma \>$ polarized, but for any given partition $\mathcal A$, there is a partition $\mathcal A'$ witnessing polarization for $\mathcal A$ with $|\mathcal A'| \leq 1+ 2^{2^{2 \cdot |\mathcal V|^2}}$. 

Having explicit bounds like these, rather than mere existence statements, does not much affect the proof of the main result of \cite{brian2024does}. 
It is relevant here, however, because it means that the statement ``the virtual refinement $(\phi,\<\mathcal V,\to\>)$ is polarized'' can be expressed not with an unbounded quantification over all partitions of $\pnmf$, but rather only by quantifying over partitions of a certain specific size. And roughly speaking, unbounded quantification over all partitions is second-order in the language of dynamical systems, but bounded quantification over partitions of a certain size is first-order. 

Specifically, note that for any given $n$, ``there is a partition of size $n$ with some first-order property'' is first-order expressible in the language of dynamical systems. Thus ``there is a partition of size $\leq N$ with some first-order property'' is also first-order expressible. In light of the previous paragraph, this means that for every virtual refinement $(\phi,\<\mathcal V,\to\>)$ of a partition $\mathcal A$, the statement ``the virtual refinement $(\phi,\<\mathcal V,\to\>)$ is polarized'' can be expressed by a sentence of first-order logic, provided that ``is polarized'' is also first-order expressible. 
And it is, because while this assertion involves a quantification over partitions, it is not difficult to check that two virtual refinements $(\phi,\<\mathcal V,\to\>)$ and $(\psi,\<\mathcal V',\to'\>)$ are compatible if and only if there is a partition of size at most $|\mathcal V| \cdot |\mathcal V'|$ witnessing their compatibility.

Consequently, the statement that ``every virtual refinement of every partition of $\pnmf$ is polarized'' is expressible as a scheme of first-order sentences (one sentence for every virtual refinement). If $\< \mathbb A,\alpha \>$ is elementarily equivalent to $\< \pnmf,\s \>$, then it too satisfies this first-order scheme. 
Because each particular sentence in this scheme is sufficiently absolute (see the proof of Theorem 11.3 in \cite{brian2024does}), it follows that every virtual refinement of every partition of $\mathbb A$ is polarized. 
By Theorem 9.7 in \cite{brian2024does}, this means $(\< \mathbb A,\alpha \>,\< \mathbb B,\beta \>,\iota,\eta)$ has a solution.
\end{proof}

The proof of Theorem 12.3 in \cite{brian2024does} relies on the Lifting Lemma, Lemma 6.2. If we substitute the use of Lemma 6.2 in that proof with the previous claim instead (which strengthens Lemma 6.2), then we get a strengthening of Theorem 12.3, namely: \ch implies that if $\< \mathbb A,\s \>$ is a countable substructure of $\fA_\s$ and $\< \mathbb A,\s \> \equiv \fA_\s$, then any embedding of $\< \mathbb A,\s \>$ into $\fA_\s$ extends to an automorphism of $\fA_\s$. In other words,~$\fA_\s$ is countably homogeneous. 

This domino falls into others. 
The proof of Theorem~12.4 in \cite{brian2024does} relies on Theorem~12.3, and if we substitute the stronger version of Theorem~12.3 from the previous paragraph, we get a strengthening of Theorem 12.4, namely: every countable substructure of $\fA_\s$ elementarily equivalent to $\fA_\s$ is an elementary substructure. 
(Recall that, generally speaking, a model can have a substructure satisfying the same theory but failing to be an elementary substructure. But apparently this cannot happen with $\fA_\s$.)

In like manner, the proof of Theorem 12.5 in \cite{brian2024does} relies on Theorem 12.4, and if we substitute the stronger version from the previous paragraph, we get a strengthening of Theorem 12.5, namely: $\fA_\s$ is model complete. 

Furthermore, $\fA_\s$ is countably universal. 
This follows from the model completeness of $\fA_\s$ and Theorem 7.9 in \cite{brian2024does}, which implies that every countable model of $\mathrm{Th}(\fA_\s)$ embeds into $\< \pnmf,\sigma \>$. 

Thus $\fA_\s$ is model complete, countably universal, and countably homogeneous.  The latter two properties together imply $\fA_\s$ is countably saturated.
\end{proof}

It would be desirable to find a direct proof that $\fA_\sigma$ and $\fA_{\sigma^{-1}}$ are elementarily equivalent and countably saturated. As this would immediately imply the main result of \cite{brian2024does}, such a proof would undoubtedly be difficult (and illuminating). 

\begin{theorem}\label{thm:Saturation}
If $f$ is an almost permutation of $\N$, then $\fA_f$ is countably saturated if and only if $\mathrm{Z}(f) \neq s_{\Z \times \Z}$. 
\end{theorem}

\begin{proof}  
By Lemma~\ref{lem:TriviallyConjugate}, we may replace $f$ with any equivalent almost permutation without changing the structure $\fA_f$. 
Thus, applying Lemma~\ref{lem:OrbitStructure}, let us assume $f = \mathrm{R}(f) \oplus \mathrm{Z}(f) \oplus \mathrm{S}(f)$. 

First suppose that $\mathrm{Z}(f) \neq s_{\Z \times \Z}$.  
Because $f = \mathrm{R}(f) \oplus \mathrm{Z}(f) \oplus \mathrm{S}(f)$, it is clear that~$\fA_f$ is the direct product of a rotary automorphism $\fA_{\mathrm{R}(f)}$ and finitely many copies of $\fA_\s$ and $\fA_{\s^{-1}}$. (Recall that a single $\Z$-like orbit is equivalent to the sum of an $\N$-like orbit and a reverse $\N$-like orbit.) 
Since the direct product of finitely many countably saturated structures (naturally, of the same language) is countably saturated, by Proposition~\ref{prop:ShiftSaturation} this shows that $\fA_f$ is countably saturated. 

Next suppose that $\mathrm{Z}(f) = s_{\Z \times \Z}$. 
This means that $\fA_{s_{\Z \times \Z}}$ is a definable substructure of $\fA_f$ (because $\fA_{s_{\Z \times \Z}}$ is simply the restriction of $\alpha_f$ to $[\mathrm{dom}(\mathrm{Z}(f))]$). By Proposition~\ref{prop:Saturation}, this means that $\fA_f$ has a definable substructure that is not countably saturated. Hence $\fA_f$ is not countably saturated. 
\end{proof}

%%%%%%%%%%
\section{The theory of a trivial automorphism}\label{sec:RotaryTheory}
%%%%%%%%%%

Results of the previous section show that potential isomorphism between dynamical systems associated with trivial automorphisms of $\pnmf$ reduces to elementary equivalence under \ch. This raises the question of what properties of such $\alpha$ are first order. In this section we provide some partial answers to this question.  
We shall be abusing the language and writing ``the theory of an automorphism $\alpha$'' instead of ``the theory of the dynamical system $\fA_\alpha$''.
Following \cite{Darji}, \cite{brian2024does}, or \cite{GoodMeddaugh}, given an automorphism $\alpha: \pnmf \to \pnmf$ and a partition of unity $\mathcal A$ for $\pnmf$, we can represent the action of $\alpha$ on $\mathcal A$ as a directed graph. The edge relation for this graph is the ``hitting relation''
\[
a \toa b \quad \Leftrightarrow \quad \alpha(a) \wedge b \neq [\emptyset]_\mathrm{Fin}.
\]
Such a digraph is said to be \emph{represented} in $\fA_\alpha$. 
This definition and the following lemma apply to automorphisms that are not necessarily trivial. 

\begin{lemma}\label{L.digraphs} Suppose that  $\alpha$ and $\beta$ are two automorphisms of $\pnmf$. 
	\begin{enumerate}
		\item  $\ThE(\fA_\alpha)\subseteq\ThE(\fA_\beta)$ if and only if every digraph represented in $\fA_\alpha$ is represented in $\fA_\beta$.
		
		\item  $\fA_\alpha$ and $\fA_\beta$ have the same existential theory if and only if the same digraphs are represented in  $\fA_\alpha$ and  $\fA_\beta$.
	\end{enumerate}
\end{lemma}
\begin{proof} It suffices to prove the first part. 
	A digraph being represented in $\fA_\alpha$ by some partition $\mathcal A = \{a_1,a_2,\dots,a_n\}$ can be expressed by an existential sentence: 
	$\exists a_1,a_2,\dots,a_n$ such that $a_i \wedge a_j = [\emptyset]_\mathrm{Fin}$ whenever $i \neq j$, $\bigvee_{i \leq n}a_i = [\N]_\mathrm{Fin}$, and $\alpha(a_i) \wedge a_j \neq [\emptyset]_\mathrm{Fin}$ if and only if $a_i\toa a_j$ for all $i,j\leq n$.  
	
	It therefore suffices to prove the converse implication. Fix an existential sentence  $\exists x_1,x_2,\dots,x_n \ \varphi(x_1,x_2,\dots,x_n)$, where $\varphi$ is quantifier-free, in the laguage of dynamical systems. Assume that this sentence is satisfied in $\fA_\alpha$ by some $b_1,\dots, b_n$. 
Let $\cA$ be the (finite) partition of unity generated by $\alpha^i(b_j)$ for all $j\leq n$ and $i$ such that $\alpha^i(x_j)$ occurs in $\varphi$. Consider the digraph associated with $\cA$. For each $j\leq n$ let $\cA_j$ be the set of pieces in $\cA$ such that $b_j=\bigvee\cA_j$. Assume that this digraph is represented in some $\fA_\beta$, by a partition $\cA'$, and let $f\colon \cA\to \cA'$ be a digraph isomorphism.  Then the elements $c_j=\bigvee f[\cA'_j]$ for $j\leq n$  satisfy $\varphi$ in $\fA_\beta$. 
\end{proof}

In analogy with almost permutations, let us say a function $f: \w \to \w$ is an \emph{almost surjection} if $\w \setminus \mathrm{image}(f)$ is finite.

\begin{lemma}\label{lem:RotaryEmbedding}
Let $\bar m = \seq{m_j}{j \in \w}$ and $\bar n = \seq{n_j}{j \in \w}$ be sequences in $\N$, and suppose there is a finite-to-one almost surjection $f: \w \to \w$ such that $m_{f(j)} \!\mid\! n_j$ for all but finitely many $j \in \w$. Then $\fA_{\bar m}$ embeds in $\fA_{\bar n}$, and therefore $\mathrm{Th}_\exists(\fA_{\bar m}) \subseteq \mathrm{Th}_\exists(\fA_{\bar n})$.
\end{lemma}
\begin{proof}
Fix $\bar m$, $\bar n$, and $f$ as described. 
Let $I_k = \big[ \sum_{i=1}^{k-1} m_i,\sum_{i=1}^k m_i \big)$ for all $k$, so that $I_0, I_1, I_2,\dots$ is the sequence of intervals defined for $\bar m$ as in Definition~\ref{def:RotaryMap}, and let $J_k = \big[ \sum_{j=1}^{k-1} n_j,\sum_{j=1}^k n_j \big)$ for all $k$, so that $J_0, J_1, J_2,\dots$ is the sequence of intervals defined for $\bar n$ as in Definition~\ref{def:RotaryMap}. 
Our condition on $f$ asserts that for some $k_0 \in \N$, if $k \geq k_0$ then there is some $d_k \in \N$ such that $|J_k| = d_k|I_{f(k)}|$.

Define a partial function $e: \N \to \N$ as follows: $\mathrm{dom}(e) = \bigcup_{k \geq k_0} J_k$, and 
$$e(j) \,=\, \min I_{f(k)} + j'  \ \ \text{ when $j \in J_k$ and $0 \leq j' < |I_{f(k)}|$ with $j = j' \ (\mathrm{mod} \ |I_{f(k)}|)$}.$$
In other words, if we think of the $I_k$ and $J_k$ as cyclically ordered, then $e$ is simply the function that takes each $J_k$ and wraps it precisely $d_k$ times around $I_k$. 
Observe that $e$ is finite-to-one, because $f$ is: the preimage of a point in $I_\ell$ has size $\sum_{k \in f^{-1}(\ell)} d_k$. 
Observe also that $\mathrm{dom}(e)$ is a cofinite subset of $\N$.

Define a function $\eta: \pnmf \to \pnmf$ by setting
$$\eta([A]_\mathrm{Fin}) \,=\, \big[e^{-1}[A]\big]_\mathrm{Fin}$$
for all $A \sub \N$. 
Note that $\eta$ is well-defined because $e$ is finite-to-one. 
Because function pullbacks preserve unions and intersections, $\eta$ preserves the Boolean-algebraic operations on $\pnmf$. 
Furthermore, because the domain of $e$ is cofinite in $\N$, if $[A]_\mathrm{Fin} \neq [B]_\mathrm{Fin}$ then $\big[e^{-1}[A]\big]_\mathrm{Fin} \neq \big[e^{-1}[B]\big]_\mathrm{Fin}$; i.e., $\eta$ is injective. Hence $\eta$ is a Boolean-algebraic embedding $\pnmf \to \pnmf$. 

Because the permutation $f_{\bar m}$ cyclically permutes the $I_k$, and likewise $f_{\bar n}$ cyclically permutes the $J_k$, we have $e \circ f_{\bar n}(i) = f_{\bar m} \circ e(i)$ for all $i \in \mathrm{dom}(e)$. From this, and our definition of $\alpha_{\bar m}$ and $\alpha_{\bar n}$, it follows that $\eta \circ \alpha_{\bar n} = \alpha_{\bar m} \circ \eta$. 

Thus $\eta$ is an embedding of $\fA_{\bar m}$ into $\fA_{\bar n}$. 
The second assertion of the lemma, that $\mathrm{Th}_\exists(\fA_{\bar m}) \subseteq \mathrm{Th}_\exists(\fA_{\bar n})$, follows immediately from this, because existential statements are upwards absolute.
\end{proof}

\begin{lemma}\label{lem:CentralCover}
Let $\alpha$ be a rotary automorphism. For every $x \in \pnmf$ there is a unique $\mathrm{cov}(x) \in \FIX_\alpha$ such that
\begin{enumerate}
\item $x \leq \mathrm{cov}(x)$ and
\item if $y \in \FIX_\alpha$ and $x \leq y$, then $\mathrm{cov}(x) \leq y$.
\end{enumerate} 
In other words, $c(x)$ is the least member of $\FIX_\alpha$ above $x$. 
\end{lemma}
\begin{proof}
Fix $\bar n = \seq{n_j}{j \in \w}$ such that $\alpha = \alpha_{\bar n}$. 
Let $I_k = \big[ \sum_{i=1}^{k-1} m_i,\sum_{i=1}^k m_i \big)$ for all $k$, so that $I_0, I_1, I_2,\dots$ is the sequence of intervals defined for $\bar m$ as in Definition~\ref{def:RotaryMap}. 
Given $x = [X]_\mathrm{Fin} \in \pnmf$, let
$$\mathrm{cov}(X) \,=\, \textstyle \bigcup \set{I_k}{I_k \cap X \neq \emptyset}$$
and let $\mathrm{cov}(x) = [\mathrm{cov}(X)]_\mathrm{Fin}$. Because each $I_k$ is finite, this definition of $\mathrm{cov}(x)$ is independent of the representative in $[X]_\mathrm{Fin}$ used to define it.

Because $f_{\bar n}$ cyclically permutes each of the $I_k$, $\mathrm{cov}(x) \in \FIX_\alpha$. 
Now suppose $Y \subseteq \N$ and $x \leq [Y]_\mathrm{Fin} \in \FIX_\alpha$. Adding finitely many elements to $Y$ if needed, we may (and do) assume $X \subseteq Y$. 

For each $k$ such that $\emptyset \neq Y \cap I_k \neq I_k$, fix some $i_k$ such that $i_k \in Y$ and $f_{\bar n}(i_k) \notin Y$. If there were an infinite set $S$ of integers $k$ with this property, then we would have $[\set{i_k}{k \in S}]_\mathrm{Fin} \leq [Y]_\mathrm{Fin}$ and 
$$\alpha([\set{i_k}{k \in S}]_\mathrm{Fin}) = [\set{f_{\bar n}(i_k)}{k \in S}]_\mathrm{Fin} \not\leq [Y]_\mathrm{Fin},$$
which would mean $[Y]_\mathrm{Fin} \notin \FIX_\alpha$. But $[Y]_\mathrm{Fin} \in \FIX_\alpha$, so there are only finitely many $k$ with $\emptyset \neq Y \cap I_k \neq I_k$. It follows that $Y$ contains all but finitely many elements of $\mathrm{cov}(X)$; i.e., $\mathrm{cov}(x) = [\mathrm{cov}(X)]_\mathrm{Fin} \leq [Y]_\mathrm{Fin}$.
\end{proof}

\begin{lemma}\label{lem:RotaryObstruction}
Let $m \in \N$ and $j < m$. 
There is a first-order sentence $\varphi$ in the language of dynamical systems such that, 
for any sequence $\bar n = \seq{n_k}{k \in \w}$ in $\N$ with $n_k \geq m$ for all but finitely many $k$, 
$\fA_{\bar n} \models \varphi$ if and only if $n_k = j \ (\mathrm{mod} \ m)$ for all but finitely many $k \in \w$.
\end{lemma}
\begin{proof}
Fix $m \in \N$ and $j \in \{0,1,\dots,m-1\}$. 
Let $\bar n = \seq{n_k}{k \in \w}$ be a sequence in~$\N$ with $n_k \geq m$ for all but finitely many $k$. 

Let $\alpha$ be a variable to denote an arbitrary automorphism of $\pnmf$. 
We aim to describe a first-order formula with $\alpha$ as a free variable that is satisfied by $\alpha = \alpha_{\bar n}$ if and only if $n_k = j \ (\mathrm{mod} \ m)$ for all but finitely many $k \in \w$. 

For each $x \in \pnmf$, let $\mathrm{cov}(x)$ denote the least member of $\FIX_\alpha$ above $x$, as described in the previous lemma. This function is first-order definable:
$$y = \mathrm{cov}(x) \ \ \Leftrightarrow \ \ x \leq y, \, \alpha(y) = y, \text{ and if } x \leq z \text{ and } \alpha(z)=z \text{ then } y \leq z.$$
Let us say $x \in \pnmf$ is \emph{small} if $\mathrm{cov}(y) < \mathrm{cov}(x)$ whenever $y < x$. 
Because $\mathrm{cov}$ is definable, ``$x$ is small'' is expressible in a first-order way.

Now consider the first-order sentence expressing the following: 

\begin{itemize}
\item[$\varphi:\,$] There exist $b_0,b_1,\dots,b_{m-1},a_0,a_1,\dots,a_{j-1} \in \pnmf$ such that
\begin{itemize}
\item[$\circ$] $\{b_0,\dots,b_{m-1},a_0,\dots,a_{j-1}\}$ is a finite partition of unity for $\pnmf$. Explicitly, this says $b_0,\dots,b_{k-1},a_0,\dots,a_{j-1} \neq [\emptyset]_\mathrm{Fin}$, and 
$$b_0 \vee \dots \vee b_{m-1} \vee a_0 \vee \dots \vee a_{j-1} = [\N]_\mathrm{Fin},$$
and if $x,y \in \{b_0,\dots,b_{m-1},a_0,\dots,a_{j-1}\}$ and $x \neq y$, then $x \wedge y = [\emptyset]_\mathrm{Fin}$. 
\item[$\circ$] $\mathrm{cov}(x) = [\N]_\mathrm{Fin}$ for all $x \in \{b_0,\dots,b_{m-1},a_0,\dots,a_{j-1}\}$.
\item[$\circ$] $a_0,a_1,\dots,a_{j-1}$ are small.
\item[$\circ$] If $\ell \neq m-1$ then $\alpha(b_\ell) = b_{\ell+1}$.
\item[$\circ$] $\alpha(a_\ell) = a_{\ell+1}$ for all $\ell = 0,1,\dots,j-2$.
\item[$\circ$] $\alpha(b_{m-1}) = a_0 \vee b_0$.
\item[$\circ$] $\alpha(a_{j-1}) \vee \alpha(b_{m-1}) = b_0$.
\end{itemize}
\end{itemize}
The last four bullet points in this description can be visualized via the kind of digraph described at the beginning of this section. Recall that an arrow from $x$ to $y$ indicates that $\alpha(x) \wedge y \neq [\emptyset]_\mathrm{Fin}$.

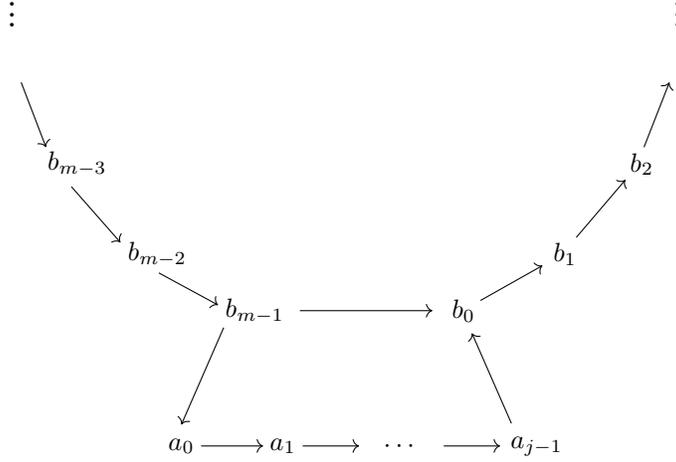
\begin{center}
\begin{tikzpicture}[scale=.75]

%\draw[black!20] (0,6) circle (6cm);

\node at (-4.75,3) {$b_{m-3}$};
\node at (-3.33,1.4) {$b_{m-2}$};
\node at (-1.6,.4) {$b_{m-1}$};

\node at (2.1,.4) {$b_0$};
\node at (3.9,1.4) {$b_1$};
\node at (5.26,3) {$b_2$};

\node at (5.9,5.8) {\LARGE $\vdots$};
\node at (-5.9,5.8) {\LARGE $\vdots$};

\draw[->] (-.8,.4) -- (1.55,.4);
\draw[->] (2.4,.58) -- (3.5,1.2);
\draw[->] (4.1,1.7) -- (4.97,2.72);
\draw[->] (5.3,3.3) -- (5.75,4.45);

\draw[->] (-3.3,1.07) -- (-2.25,.5);
\draw[->] (-4.85,2.6) -- (-3.98,1.61);
\draw[->] (-5.74,4.45) -- (-5.3,3.3);

\node at (-2.9,-2) {$a_0$};
\node at (-1.1,-2) {$a_1$};
\node at (1,-2) {$\cdots$};
\node at (3.4,-2) {$a_{j-1}$};

\draw[->] (-2.55,-2) -- (-1.45,-2);
\draw[->] (-.75,-2) -- (.25,-2);
\draw[->] (1.75,-2) -- (2.75,-2);

\draw[->] (-2.15,.1) -- (-2.9,-1.62);
\draw[<-] (2.25,0) -- (2.95,-1.6);

\end{tikzpicture}
\end{center}
We claim that this formula $\varphi$ is satisfied by $\alpha_{\bar n}$ if and only if $n_k = j \ (\mathrm{mod} \ m)$ for all but finitely many $k \in \w$. 
Let $I_k = \big[ \sum_{i=1}^{k-1} n_i,\sum_{i=1}^k n_i \big)$ for all $k$, so that $I_0, I_1, I_2,\dots$ is the sequence of intervals defined for $\bar n$ as in Definition~\ref{def:RotaryMap}. 

First, suppose that $n_k = j \ (\mathrm{mod} \ m)$ for all but finitely many $k \in \w$. 
To see that $\varphi$ is satisfied by $\alpha_{\bar n}$, define
\begin{align*}
A_\ell &\,=\, \set{\max I_k-j+\ell}{k \in \w} \ \ \ \quad \qquad \qquad \qquad \qquad \qquad \qquad \qquad \text{ for all }\ell < j, \\
B_\ell &\,=\, \textstyle \bigcup_{k \in \w}\set{i \in I_k}{i-\min I_k = \ell \ (\mathrm{mod} \ m) \text{ and } i \leq \max I_k - j} \ \  \text{ for all }\ell < m. 
\end{align*}
In other words, put the last $j$ elements of $I_k$ into $A_0,\dots,A_{j-1}$, and then form $B_0,\dots,B_{m-1}$ by partitioning the remainder of $I_k$ into its mod-$m$ equivalence classes. 
For each $\ell < j$, let $a_\ell = [A_\ell]_\mathrm{Fin}$, and for each $\ell < m$ let $b_\ell = [B_\ell]_\mathrm{Fin}$. We claim that $a_0,\dots,a_{j-1},b_0,\dots,b_{m-1}$ witness $\varphi$ for $\fA_{\bar n}$.

The sets $A_0,\dots,A_{j-1},B_0,\dots,B_{m-1}$ are all nonempty (for the $B_\ell$, this is because $n_k \geq m$ for all but finitely many $k$), and these sets form a partition of $\N$. 
It follows that $\{a_0,\dots,a_{j-1},b_0,\dots,b_{m-1}\}$ is a finite partition of unity in $\pnmf$. 
The proof of Lemma~\ref{lem:CentralCover} shows that if $X \subseteq \N$, then $\mathrm{cov}([X]_\mathrm{Fin}) = \big[ \bigcup \set{I_k}{X \cap I_k \neq \emptyset} \big]_\mathrm{Fin}$. 
Because each of the $A_\ell$ and the $B_\ell$ contains at least one member of all but finitely many of the $I_k$, it follows that 
$$\mathrm{cov}(a_0) = \dots = \mathrm{cov}(a_{j-1}) = \mathrm{cov}(b_0) = \dots = \mathrm{cov}(b_{m-1}) = [\N]_\mathrm{Fin}.$$ 
%Each of the $B_\ell$ contains many members of all but finitely many of the $I_k$ (because $\lim_{i \to \infty}n_i = \infty$), and so each of the $B_\ell$ can be split into two smaller sets $X$ and $Y$ with $\mathrm{cov}([X]_\mathrm{Fin}) = \mathrm{cov}([Y]_\mathrm{Fin}) = [\N]_\mathrm{Fin}$. Thus $b_0,b_1,\dots,b_{m-1}$ are not small. 
Each of the $A_\ell$ contains exactly one member of all but finitely many of the $I_k$, so if $X \subseteq \N$ and $[X]_\mathrm{Fin} < [A_\ell]_\mathrm{Fin}$, then $X$ misses infinitely many of the $I_k$, and this means $\mathrm{cov}([X]_\mathrm{Fin}) < [\N]_\mathrm{Fin}$. 
Thus $a_0,a_1,\dots,a_{j-1}$ are small.  
Finally, it is clear from the definitions of the $A_\ell$ and $B_\ell$ that the last four bullet points in the definition of $\varphi$ are satisfied. 
Hence $\fA_{\bar n} \models \varphi$. 

Next suppose that $\fA_ {\bar n} \models \varphi$. 
Let $A_0,A_1,\dots,A_{j-1},B_0,B_1,\dots,B_{m-1}$ be subsets of $\N$ such that setting $a_\ell = [A_\ell]_\mathrm{Fin}$ for all $\ell < j$ and $b_\ell = [B_\ell]_\mathrm{Fin}$ for all $\ell < m$ gives a witness to $\varphi$. 
As in the previous paragraph, we again use the fact that if $X \subseteq \N$ then $\mathrm{cov}([X]_\mathrm{Fin}) = \big[ \bigcup \set{I_k}{X \cap I_k \neq \emptyset} \big]_\mathrm{Fin}$. 
Because of this, all of the $A_\ell$ and all of the $B_\ell$ must meet all but finitely many of the intervals $I_k$. 
Furthermore, because each of the $a_\ell$ is small, each $A_\ell$ must meet all but finitely many of these intervals in exactly one point. (This is because, if $A_\ell$ were to meet infinitely many intervals in more than one point, we could split $A_\ell$ into two sets $X$ and $Y$ each meeting all but finitely many of the $I_k$; but then $[X]_\mathrm{Fin} < [A_\ell]_\mathrm{Fin} = d$ and $\mathrm{cov}([X]_\mathrm{Fin}) = [\N]_\mathrm{Fin} = \mathrm{cov}([A_\ell]_\mathrm{Fin})$, contradicting our assumption that $a_\ell$ is small.) 
The final four conditions in the definition of $\varphi$ imply: 
\begin{itemize}
\item[$\circ$] If $\ell \neq m-1$, then $f_{\bar n}[B_\ell] =^* B_{\ell+1}$.
\item[$\circ$] $f_{\bar n}[A_\ell] =^* A_{\ell+1}$ for all $\ell = 0,1,\dots,j-2$.
\item[$\circ$] $f_{\bar n}[B_{m-1}] =^* A_0 \cup B_0$.
\item[$\circ$] $f_{\bar n}[A_{j-1}] \cup f_{\bar n}[B_{m-1}] =^* B_0$.
\end{itemize}
Recall that $f_{\bar n} \restriction I_k$ is a cyclic permutation of size $n_k$. 
For the last part of the proof, if $X \subseteq I_k$ then let us write $X+1$ for $f_{\bar n}[X]$. Combining the above observations, the following must be true for all but finitely many $k$: 

\begin{itemize}
\item[$\!$] There are some $a_0,a_1,\dots,a_{j-1} \in I_k$ and $B^k_0,B^k_1,\dots,B^k_{m-1} \subseteq I_k$ such that
\begin{itemize}
\item[$\circ$] If $\ell \neq m-1$, then $B^k_\ell+1 = B^k_{\ell+1}$.
\item[$\circ$] $a_\ell+1 = a_{\ell+1}$ for all $\ell = 0,1,\dots,j-2$.
\item[$\circ$] $B^k_{m-1}+1 = \{a_0\} \cup B^k_0$.
\item[$\circ$] $\{a_{j-1}+1\} \cup (B^k_{m-1}+1) = B^k_0$.
\end{itemize}
\end{itemize}
The first condition implies that if $\ell \neq m-1$, then every member of $B^k_\ell$ is immediately followed by a member of $B^k_{\ell+1}$, and every member of $B^k_{\ell+1}$ is immediately preceded by a member of $B^k_\ell$. It follows that $|B^k_\ell| = |B^k_{\ell+1}|$ for all $\ell \neq m-1$. 
But then 
$|B^k_0| = |B^k_1| = \dots = |B^k_{m-2}| = |B^k_{m-1}|,$ 
i.e., all the $B^k_\ell$ have the same size. 
Because $I_k = \{a_0,\dots,a_{j-1}\} \cup \bigcup_{\ell < m}B^k_\ell$, it follows that $|I_k| = j \ (\mathrm{mod} \ m)$. 
\end{proof}

\begin{theorem}\label{thm:RotaryObstruction}
For any $m \in \N$ and $j < m$, there are first-order sentences $\varphi$ and $\psi$ in the language of dynamical systems such that, for any sequence $\bar n = \seq{n_k}{k \in \w}$, 

\begin{itemize}
\item[$\circ$] $\fA_{\bar n} \models \varphi$ if and only if $n_k = j \ (\mathrm{mod} \ m)$ for all but finitely many $k \in \w$. 
\item[$\circ$] $\fA_{\bar n} \models \psi$ if and only if $n_k = j \ (\mathrm{mod} \ m)$ for infinitely many $k \in \w$.
\end{itemize}
\end{theorem}
\begin{proof}
Fix $m \in \N$ and $j \in \{0,1,\dots,m-1\}$. 
Let $\bar n = \seq{n_k}{k \in \w}$ be a sequence in $\N$ with $n_k \geq m$ for all but finitely many $k$. 

%To describe $\varphi$, let us consider three possibilities. 
Supposing we knew beforehand that $n_k \geq m$ for all but finitely many $k$, then the required formula, let us call it $\varphi_{\geq m}$, is given by Lemma~\ref{lem:RotaryObstruction}. 
On the other hand, supposing we knew beforehand that $n_k < m$ for all but finitely many $k$, then the required formula is: 
\begin{align*}
\forall x \ \alpha^j(x) = x \ \text{ and } \ \forall y \,\exists z \leq y \ (\alpha^0(z) \neq z \text{ or } \alpha^1(z) \neq z \text{ or } \dots \text{ or } \alpha^{j-1}(z) \neq z).
\end{align*}
Call this formula $\varphi_{<m}$. It is not difficult to see that $\varphi_{<m}$ is true for some $\fA_{\alpha_{\bar n}}$ if and only if $\bar n$ is eventually equal to the constant sequence $\< j,j,j,\dots \>$. 

For the general case, we first need a definition. 
For each $c \in \FIX_\alpha$, and each formula $\chi$ in the language of dynamical systems, let $\chi^c$ denote the relativization of the formula $\chi$ to $c \!\downarrow\ = \set{x \in \pnmf}{x \leq c}$. 
Specifically, we say that the formula $\chi^c$ is satisfied by the dynamical system $\<\pnmf,\alpha\>$ if and only if the formula $\chi$ is satisfied by the possibly smaller dynamical system $\<c \!\downarrow\,,\,\alpha\! \restriction\! (c \!\downarrow)\>$. We allow for the possibility that $c=[\emptyset]_{\Fin}$ and set  $\varphi^{[\emptyset]_{\Fin}}$ to be true for all $\varphi$. 

Let $\varphi$ be the following formula:
$$\exists c,d \in \FIX_\alpha \ c \vee d = 1 \text{ and } \varphi_{<m}^c \text{ and } \varphi_{\geq m}^d.$$
Note that this formula does not insist either $c$ or $d$ is nonzero, so it is implied by either $\varphi_{<m}$ or $\varphi_{\geq m}$, by taking $c$ or $d$, respectively, to be $[\N]_\mathrm{Fin}$. 
We claim that this formula $\varphi$ is satisfied by $\alpha_{\bar n}$ if and only if $n_k = j \ (\mathrm{mod} \ m)$ for all but finitely many $k \in \w$. 
Let $I_k = \big[ \sum_{i=1}^{k-1} n_i,\sum_{i=1}^k n_i \big)$ for all $k$, so that $I_0, I_1, I_2,\dots$ is the sequence of intervals defined for $\bar n$ as in Definition~\ref{def:RotaryMap}. 

First, suppose that $n_k = j \ (\mathrm{mod} \ m)$ for all but finitely many $k \in \w$. Let $C = \bigcup \set{I_k}{|I_k| \leq m}$ and $D = \bigcup \set{I_k}{|I_k| > m}$; let $c = [C]_\mathrm{Fin}$ and $d = [D]_\mathrm{Fin}$. Then $\varphi_{<m}^c$ and $\varphi^d_{\geq m}$ both hold, and therefore $\fA_{\bar n} \models \varphi$. 
Conversely, suppose that $\fA_{\bar n} \models \varphi$. This means there are $c,d \in \FIX_{\alpha_{\bar n}}$ such that $\varphi_{< m}^c$ and $\varphi_{\geq m}$ both hold. 
Fix $C,D \subseteq \N$ such that $c = [C]_\mathrm{Fin}$ and $d = [D]_\mathrm{Fin}$. 
Now, $c,d \in \FIX_{\alpha_{\bar n}}$ means $\mathrm{cov}(c) = c$ and $\mathrm{cov}(d) = d$, and the proof of Lemma~\ref{lem:CentralCover} then shows that, up to a finite error, $C$ and $D$ are each equal to the union of some subset of $\set{I_k}{k \in \w}$. Modifying $C$ and $D$ on a finite set, we may (and do) assume $C$ and $D$ are each equal to the union of some subset of $\set{I_k}{k \in \w}$: say $C = \set{I_k}{k \in \bar C}$ and $C = \set{I_k}{k \in \bar D}$.  
Because $\varphi_{< m}^c$ holds, $|I_k| = j$ for all but finitely many $k \in \bar C$. Likewise, because $\varphi^d_{\geq m}$ holds, $|I_k| \geq m$ and $|I_k| = j \ (\mathrm{mod}\ m)$ for all but finitely many $k \in \bar D$. Thus $|I_k| = j \ (\mathrm{mod}\ m)$ for all but finitely many $k \in \w$. 

This completes the proof of the first part of the theorem. The second part is a relatively straightforward consequence of the first part, making another use of the idea of relativization. Let $\psi$ be the formula:
$$\exists c \in \FIX_\alpha \ c \neq [\emptyset]_\mathrm{Fin} \text{ and } \varphi^c.$$
Because $\varphi$ holds for some $\alpha_{\bar n}$ if and only if $|I_k| = j \ (\mathrm{mod}\ m)$ for all but finitely many $k \in \w$, it is not difficult to see that $\psi$ holds for $\alpha_{\bar n}$ if and only if $|I_k| = j \ (\mathrm{mod}\ m)$ for infinitely many $k \in \w$. 
\end{proof}

\begin{corollary}\label{cor:ExistentialTheory}
There are two trivial automorphisms $\alpha$ and $\beta$ of $\pnmf$ such that $\fA_\alpha$ and $\fA_\beta$ are biembeddable, so that $\mathrm{Th}_\exists(\fA_{\alpha}) = \mathrm{Th}_\exists(\fA_{\beta})$, but $\mathrm{Th}(\fA_\alpha) \neq \mathrm{Th}(\fA_\beta)$.
\end{corollary}
\begin{proof} 
Define $\bar m = \seq{m_i}{i \in \w}$ and $\bar n = \seq{n_j}{j \in \w}$ by taking $m_i = 2^{2i}$ and let $n_j = 2^{2j+1}$. We claim that the rotary maps $\alpha_{\bar m}$ and $\alpha_{\bar n}$ witness the theorem. 

On the one hand, taking $f = \mathrm{id}_\omega$ in Lemma~\ref{lem:RotaryEmbedding} shows $\fA_{\bar m}$ embeds in $\fA_{\bar n}$, and $\mathrm{Th}_\exists(\fA_{\bar m}) \subseteq \mathrm{Th}_\exists(\fA_{\bar n})$. 
On the other hand, reversing the roles of $\bar m$ and $\bar n$ in Lemma~\ref{lem:RotaryEmbedding} and taking $f(n) = n-1$ (with $f(0)$ defined arbitrarily) shows $\fA_{\bar n}$ embeds in $\fA_{\bar m}$, and $\mathrm{Th}_\exists(\fA_{\bar n}) \subseteq \mathrm{Th}_\exists(\fA_{\bar m})$. 

Finally, observe that $m_i = 1 \ (\mathrm{mod}\ 3)$ and $n_i = 2 \ (\mathrm{mod}\ 3)$ for all $i \in \w$. 
By Lemma~\ref{lem:RotaryObstruction}, $\fA_{\bar m}$ and $\fA_{\bar n}$ are not elementarily equivalent. 
\end{proof}

This corollary reveals a failure of the Cantor--Schr\"{o}der--Bernstein property for rotary automorphisms: it is possible that two such structures embed in each other, but that they are not isomorphic.

If two complete theories of the same language are model-complete and have the same set of existential consequences, then they are equal. This is proven by a standard sandwich/tower argument. Therefore  Corollary~\ref{cor:ExistentialTheory}
 immediately implies the following. 

\begin{corollary}\label{cor:NotMC}
There is a trivial automorphism $\alpha$ of $\pnmf$ such that $\mathrm{Th}(\fA_\alpha)$ is not model complete. \qed
\end{corollary}

We do not have a concrete example of a trivial automorphism $\alpha$ for which the theory of $\fA_\alpha$  is not model-complete, but it is  likely that the theory of $\fA_{\bar n}$ is not model complete whenever $\limsup_i n_i=\infty$.

If $\kappa$ is an uncountable cardinal then a structure $B$ in a countable language is said to be \emph{$\kappa$-saturated} if every consistent full type of cardinality $<\kappa$ over $B$ is realized  in $B$. (Thus our `countably saturated' corresponds to `$\aleph_1$-saturated'.)
The following result is well-known but we could not find a precise reference. 

\begin{lemma} \label{L.embedding} Suppose that $A$ and $B$ are structures of the same language, $\kappa$ is an uncountable cardinal, $B$ is $\kappa$-saturated, and $|A|=\kappa$. Then $A$ embeds into $B$ if and only if the existential theory of $A$ is included in the existential theory of $B$. 
\end{lemma}

\begin{proof} The forward implication is trivial. For the converse implication, assume that the existential theory of $A$ is included in the existential theory of $B$. Let $(a_\xi)_{\xi<\kappa}$ be an enumeration of $A$ and let $\bt(\bar x)$ be its quantifier-free type. 
	By our assumption, it is consistent with the theory of $B$. We can therefore extend $\bt(\bar x)$ to a complete and consistent type $\bt^+(\bar x)$. Using $\kappa$-saturation of $B$ one can now proceed to recursively find $b_\xi\in B$ such that the type of $(b_\eta)_{\eta<\xi}$ is equal to the restriction of $\bt^+(\bar x)$ to the first $\xi$ variables. Then $a_\xi\mapsto b_\xi$ is the required embedding of $A$ into $B$. 
\end{proof}

Lemma~\ref{L.embedding}, together with Lemma~\ref{L.digraphs},  implies the following. 

\begin{proposition} If  $\alpha$ and $\beta$ are rotary maps then the following are equivalent for $\fA_\alpha$ and $\fA_\beta$. 
	\begin{enumerate}
		\item They are bi-embeddable. 
		\item They have the same existential theory. 
		\item They admit the same digraphs. \qed 
	\end{enumerate}
	\end{proposition}

%%%%%%%%%%%%
\section{Uniform Roe Algebras and Coronas} \label{S.uRc}
%%%%%%%%%%%%
Theorem~\ref{T.uRQ} below answers \cite[Question~9.5]{braga2018uniformRoe}.  
We briefly state the pertinent definitions; the reader is invited to look up \cite{RoeBook}, \cite{willett2020higher}, or the introduction to \cite{baudier2022uniform} for additional details, intuition, and background. 
\subsection*{Coarse equivalence} A map between metric spaces $f\colon X\to Y$ is \emph{coarse} if for every $m$ there is $n$ such that $d(x,x’)\leq m$ implies $d(f(x),f(x’))\leq n$. Two metric spaces $X$ and $Y$ are said to be \emph{coarsely equivalent} if there are coarse maps $f\colon X\to Y$ and $g\colon Y\to X$ such that both $\sup_{x\in X} d(x,g(f(x))$ and $\sup_{y\in Y} d(y,f(g(y))$ are finite. For example, $\bbR$ and $\bbZ$ are coarsely equivalent, every bounded metric space is coarsely equivalent to a point, but  $\bbR^n$ for $n\geq 1$ are coarsely inequivalent. A metric space is \emph{uniformly locally finite} if for every $n$ the supremum of the cardinalities of $n$-balls is finite. Arguably the most important class of examples of uniformly locally finite metric spaces  is given by the graph distance on Cayley graphs on finitely generated groups. 
Every uniformly locally finite space $X$ is  clearly countable. Also, if $X$ is any countable set then  the complex Hilbert space $\ell_2(X)$ with orthonormal basis $\{\delta_x\mid x\in X\}$ is separable.

For an almost permutation $\gamma$ of $\bbN$ define a metric $d_\gamma$ on $\bbN$ as follows.  Let $(\bbN,E_\gamma)$ be the graph with $\bbN$ as the set of vertices such that $i$ and $j$ are adjacent if and only if $\gamma(i)=j$ or $\gamma(j)=i$. Then $d_\gamma$ is the graph distance on $\bbN$. The size of an $n$-ball in $X_\gamma$ is at most $2n+1$, hence $X_\gamma$ is uniformly locally finite.   If $\gamma$ is a rotary automorphism, then points in different orbits are at infinite distance. Such orbits are called `coarse components’ and we now describe a standard way to fix this issue. 
If $J_n$, for $n\in \bbN$, are orbits of $\gamma$, then if $i\in J_m$ and $j\in J_n$ with $m<n$ then we set 
\begin{equation}\label{eq.d-gamma}
d_\gamma(i,j) = \max(n, \textstyle\max_{k<n}\diam(J_{k})) .
\end{equation} 
The triangle inequality is easily verified.

\begin{lemma}
If $\gamma$ is an almost permutation such that all of its orbits $X_n$ are finite, then the metric space $(\bigsqcup_n X_n,d_\gamma)$ is uniformly locally finite. 
\end{lemma}
\begin{proof}
 Writing $B_m(j)$ for the $m$-ball of $j\in X$, we need to show that $\sup_j |B_m(j)|$ is finite for all $m\geq 1$. Fix $m$. 
  Then for $j\in \bigcup_{k\leq m} J_k$ we have $B_m(j)\subseteq \bigcup_{k\leq m } J_k$, which is a finite  union of finite sets. For $j\in \bigcup_{k>m} J_k$ let $n$ be such that $j\in J_n$. 
  Then, using  \eqref{eq.d-gamma} for the first equality, we have 
 \[
 B_m(j)=B_m(j)\cap J_n=\{\gamma^k(j)\mid -m\leq k\leq m\}, 
 \]
 hence $|B_m(j)|\leq 2m+1$. Therefore $\sup_j |B_m(j)|\leq \max(2m+1,|\bigcup_{k\leq m} J_k|)$ is finite for all $j$. 
\end{proof}

We include a proof of the following lemma (although it ought to be obvious to the experts) for reader’s convenience.\footnote{An analogous remark applies to the previous lemma.}

\begin{lemma} \label{L.coarse.rotary} If $X_{\bar m}$ and $X_{\bar n}$ are coarse spaces associated to rotary permutations of~$\bbN$,  then the following are equivalent. 
	\begin{enumerate}
		\item \label{1.coarse.rotary} $X_{\bar m}$ and $X_{\bar n}$ are coarsely equivalent. 
		\item \label{2.coarse.rotary} There are $0<K<\infty$ and an almost permutation $\gamma$ of $\bbN$ such that for all but finitely many $i$ we have 
		\[
		\frac 1K n_{\gamma(i)}\leq m_i \leq K n_{\gamma(i)}. 
		\]
	\end{enumerate}
\end{lemma}

\begin{proof} \eqref{2.coarse.rotary} implies \eqref{1.coarse.rotary}: Fix $K$ and $\gamma$. 
	Let us for a moment assume that $m_i>n_{\gamma(i)}$  for all $i\in \dom(\gamma)$. For convenience of notation, we will identify the intervals $I_i$ as in the definition of rotary maps (Definition~\ref{def:RotaryMap})  with $m_i$ or $n_i$ as appropriate.  (With the understanding that the $m_i$’s the $n_i$’s denote pairwise disjoint sets.) These are the coarse components of $X_{\bar m}$ and $X_{\bar n}$. 
	For each $i$, let $k_i$ be such that $kn_{\gamma(i)}\leq m_i<(k_i+1)n_{\gamma(i)}$. Note that $\frac 1K \leq k_i\leq K$.  Split $m_i$ into $n_{\gamma(i)}$ intervals so that each one of them has cardinality $k_i$ or $k_i+1$.   Then let $f_i\colon m_i\to n_{\gamma(i)}$ collapse the $j$-th interval to $j$ and $g_i\colon n_{\gamma(i)}\to m_i$ send $j$ to the first point of the $j$th interval. Let $f$ and $g$ be the union of the $f_i$'s and of the $g_i$’s, respectively. (The domains of these functions are cofinite; extend them to the entire space arbitrarily.) Then $\frac 1K d(x,y)\leq  d(f(x),f(y))\leq d(x,y)$  and $\frac 1K d(x,y)\leq d(g(x),g(y))\leq K d(x,y)$, therefore both $f$ and $g$ are coarse. Since $f\circ g$ is the identity and $g\circ f$ is uniformly within the distance  at most $\sup_i k_i\leq K$ from the identity, $f$ and $g$ witness coarse equivalence of $X_{\bar m}$ and $X_{\bar n}$. 
	
	If the set of $i$ such that $m_i<n_{\gamma(i)}$ is nonempty, then reverse the construction on these intervals: split $n_{\gamma(i)}$ into $m_i$ intervals of the appropriate cardinality and define  $f_i$ and $g_i$ analogously to the first case. Then $f$ and $g$ obtained by taking the unions witness that $X_{\bar m}$ and $X_{\bar n}$ are coarsely equivalent.

	\eqref{1.coarse.rotary} implies \eqref{2.coarse.rotary}: Assume, towards contradiction,  that \eqref{1.coarse.rotary} holds and 
	\eqref{2.coarse.rotary} fails.  Let $f$ and $g$ be coarse maps such that for some $K$ we have  $d(x,f(g(x))\leq K$ and $d(x,g(f(x))\leq K$ for all $x$. Coarseness of $f$  and the fact that the distances between coarse components diverge to infinity together imply  that (using the same convenient abuse of notation as in the first part of the proof) $f[m_i] $ is included in a single coarse component for all but finitely many $i$; we denote this component $n_{\tilde f(i)}(s)$.  Similarly, $g[n_i]$ is included in a single coarse component, denoted $m_{\tilde g(i)}$, for all but finitely many $i$. Since $f\circ g$ and $g\circ f$ are close to the identity map, for all but finitely many $i$ we have that $\tilde f(\tilde g(i))=i=\tilde g(\tilde f(i)$, thus $\tilde f$ and $\tilde g$ are almost permutations. 
	On the other hand, because $f$ is coarse we have that  
	\[
	L=\max(\sup_y |f^{-1}(y)|,\sup_x |g^{-1}(x)|)<\infty. 
	\]
	However, our assumption on $\bar m$ and $\bar n$ implies that $m_i/n_i\notin (1/L,L)$ for all but finitely many $i$; contradiction.  
\end{proof}

\subsection*{Uniform Roe algebras and coronas}
Let $X_\gamma=(\bbN,d_\gamma)$. The uniform Roe algebra $\cstu(X_\gamma)$ and the uniform Roe corona $\roeq(X_\gamma)$ are  defined as follows.  

On $\ell_2(\bbN)$ we have the orthonormal basis $\delta_i$, for $i\in \bbN$, and $( \xi|\eta)$ denotes the inner product of vectors $\xi,\eta$.  An operator $T$ in $\cB(\ell_2(\bbN))$, the algebra of bounded linear operators on $\ell_2(\N)$, has $\gamma$-propagation\footnote{This terminology is nonstandard;  one normally considers $\ell_2(X_\gamma)$ in place  of $\ell_2(\bbN)$  and talks about propagation instead of $\gamma$-propagation.} $m$ if for all $i$ and $j$ in $\bbN$ such that $d_\gamma(i,j)>m$ we have $\langle T\delta_i|\delta_j\rangle=0$. 

The \emph{algebraic uniform Roe algebra} $\cstu[X_\gamma]$ is the algebra of all operators of finite propagation. Its norm closure is the \emph{uniform Roe algebra} $\cstu(X_\gamma)$. 

Then $\cstu(X_\gamma)$ includes $\ell_\infty(\bbN)$ (these are the operators of propagation zero) and the ideal of compact operators, $\cK(\ell_2(\bbN))$. 
The \emph{uniform Roe corona} is the quotient 
\[
\roeq(X_\gamma)=\cstu(X_\gamma)/\cK(\ell_2(\bbN))
\]
(or in our case, if $\cK(\ell_2(\bbN)) \not\subseteq \cstu(X_\gamma)$ then $\roeq(X_\gamma)=\cstu(X_\gamma)/(\cstu(X_\gamma)\cap\cK(\ell_2(\bbN))$).

It is not difficult to prove that if $X$ and $Y$ are coarsely equivalent uniformly locally finite metric spaces, then $\cstu(X)$ and $\cstu(Y)$ are Morita equivalent (for unital \cstar-algebras this is equivalent to being \emph{stably isomorphic}, i.e., isomorphic after taking tensor product with the algebra of compact operators). The question whether  the converse holds  generated considerable activity in the past fifteen years, culminating in a positive answer given in \cite{baudier2022uniform}. We are interested in an even stronger rigidity problem.  In \cite{braga2018uniformRoe} and \cite[Theorem~1.5]{baudier2022uniform} it was proven that forcing axioms imply that for uniformly locally finite metric spaces the isomorphism of uniform Roe coronas implies coarse equivalence. The key set-theoretic result in these proofs is \cite[Theorem~9.4]{mckenney2018forcing}, asserting that every *-homomorphism from $\ell_\infty(\bbN)/c_0(\bbN)$ into the Calkin algebra has an algebraic lifting (see  the proof of Theorem~2.7 in \cite{braga2018uniformRoe}). While it is well-known and easy to prove that this lifting property fails under \ch{}, all attempts to use \ch{} to prove a non-rigidity result for uniform Roe coronas failed. The closest result was  \cite[Theorem~8.1]{braga2018uniformRoe} where it was proven  that there are locally finite metric spaces such that their Roe coronas being isomorphic is independent from \ch. More precisely, this was shown to follow directly from \cite{ghasemi2016reduced} and \cite{Gha:FDD}; our Lemma~\ref{L.subsequence}	is based on one of the ideas of the former paper. 

 Theorem~\ref{T.uRQ} below improves \cite[Theorem~8.1]{braga2018uniformRoe} where an analogous statement had been proven (or rather, shown to follow readily from \cite{ghasemi2016reduced} and \cite{Gha:FDD}) for locally finite (but not uniformly locally finite) spaces.

\begin{theorem}\label{T.uRQ}
	There are uniformly locally finite metric spaces $X$ and $Y$ such that the assertion $\roeq(X)\cong \roeq(Y)$ is independent from $\ZFC$.  
\end{theorem}
We will prove a finer result in Theorem~\ref{T.uRQ+} at the end of this section. More precisely, we will show that this result is a reformulation of    Corollary~\ref{C.many.nonconjugate} (modulo some  not completely trivial work). 

If $\gamma$ is an almost permutation of $\bbN$, then  $\alpha_\gamma$ denotes the trivial automorphism of $\pnmf$ associated with $\gamma$, to wit
 \[
\alpha_\gamma ([A]_{\Fin})=\alpha ([\gamma[A]]_{\Fin}). 
\] 
Since by the Stone and Gelfand--Naimark dualities (see \cite[Theorem~1.3.2]{Fa:STCstar}) every automorphism of $\pnmf$ uniquely extends to an automorphism of $\ell_\infty(\bbN)/c_0(\bbN)$, we will also denote this automorphism by $\alpha_\gamma$

The conditional expectation from $\cB(\ell_2)$ onto $\ell_\infty$ (that is, a unital completely positive map equal to the identity on its range, see \cite[\S 3.3]{Fa:STCstar}) is defined by 
\[
E(T)=\sum_{j} (T\delta_j|\delta_j)\delta_j. 
\]
Since $E$ sends compacts to compacts, it also defines a conditional expectation, denoted $\dot E$, from $\cQ(\ell_2)$ onto $\ell_\infty/c_0$ by 
($\pi\colon \ell_\infty(\bbN)\to \ell_\infty(\bbN)/c_0(\bbN)$ denotes the quotient map)
\[
\dot E(\dot T)=\pi(E(T)). 
\]
If $\gamma$ is a bijection between subsets of $\bbN$ then let $v_\gamma$ denote the partial isometry of $\ell_2(\bbN)$ such that 
\[
v_\gamma(\delta_i)=\delta_{\gamma(i)}
\] 
for all $i\in \dom(\gamma)$ and $v_\gamma(\delta_i)=0$ otherwise.

 The following is a special case of a lemma that can be found in \cite{RoeBook}, but we include a proof for reader's convenience. 
 
\begin{lemma} \label{L.Propagation}
		Suppose that $\gamma$ is an almost permutation of $\bbN$.  
		
		\begin{enumerate}
			\item 	
			An operator $T\in \cB(\ell_2)$ has $\gamma$-propagation $\leq m$ if and only if 
			$E(T v_\gamma^k)=0$ for all $k$ such that $|k|\geq m$. 
			\item $T\in \cstu[X_\gamma]$ if and only if 
	$E (T v_\gamma^{k})=0$ for all but finitely many $k$ 
	\item Every $T\in \cstu[X_\gamma]$  can be presented as 
	\[
	T=\sum_{k\in \bbZ} v_\gamma^k E (T v_\gamma^k)
	\]
	where only finitely many of the terms in the sum on the right-hand side are nonzero. 
			\end{enumerate}
\end{lemma}

\begin{proof} We prove all three parts simultaneously. 
	As in \cite[\S 3.3.3]{Pede:Analysis}, for $\xi$ and $\eta$ in $\ell_2$ let $\xi\odot \eta$ be the operator defined by 
	\[
	(\xi\odot \eta)(\zeta)=(\zeta|\eta)\xi. 
	\]
	If $\|\xi\|_2=\|\eta\|_2=1$ then this is a rank one partial isometry that sends $\eta$ to $\xi$. 

We will also write $p_i=\delta_i\odot \delta_i$. This is the rank-one projection to the span of $\delta_i$. 
		For $i,k\in \bbN$ we have 
		$
		(Tv_\gamma^k \delta_i|\delta_i)=(T\delta_{\gamma^k(i)}|\delta_i)$, hence 
		$E(T v_\gamma^k)=\sum_j (T\delta_{\gamma^k(i)}|\delta_i)p_i$. Since $v_\gamma^k p_i=\delta_{\gamma^k(i)} \odot \delta_i$, we have
		\[
		v_\gamma^k E(Tv_\gamma^{k}))=\sum_i (T\delta_{\gamma^{k}(i)}|\delta_i)\delta_{\gamma^k(i)}\odot \delta_i. 
		\]
		The sum on the right-hand side is clearly zero if $k$ is greater than the propagation of $T$, and this completes the proof. 
\end{proof}

\begin{lemma} \label{L.tildePhi} If $\gamma$ and $\gamma'$ are almost permutations of $\bbN$ such that the  associated automorphisms of $\pnmf$ are conjugate by an automorphism of $\pnmf$,  then $\roeq(X_\gamma)\cong \roeq(X_{\gamma'})$. 
\end{lemma}

\begin{proof} 
	Let $\Phi$ be an automorphism of $\pnmf$ such that $\Phi\circ \alpha_\gamma\circ \Phi^{-1}=\alpha_{\gamma'}$; again identify $\Phi$ with an automorphism of $\ell_\infty/c_0$.

Let $\roeq[X_\gamma]$ be the `algebraic uniform Roe corona' that is, the quotient of $\cstu[X_\gamma]$  modulo the compact operators. 	If $a\in \cB(\ell_2(\bbN))$  then we write $\dot a$ for the image of $a$ under the quotient map modulo $\cK(\ell_2(\bbN))$.
	By Lemma~\ref{L.Propagation}, every $\dot T\in \roeq[X_\gamma]$  has the form $\sum_{k\in \bbZ} v_\gamma^k \dot E (\dot T v_\gamma^{k})$.

\begin{claim}	
	The formula
		\begin{equation}\label{eq.tildePhi}
	\tilde\Phi\left(\sum_{k\in \bbZ} \dot v_\gamma^k \dot E (\dot T \dot v_\gamma^{k})\right)
	=
	\sum_{k\in \bbZ} \dot v_{\gamma'}^k \dot E (\dot T \dot v_{\gamma'}^{k})
	\end{equation}
		defines a norm preserving *-isomorphism 	between $ \roeq[X_\gamma]$ and $\roeq[X_{\gamma'}]$ that sends $\dot v_\gamma$ to $\dot v_{\gamma’}$. 
	\end{claim}

	\begin{proof} By Lemma~\ref{L.Propagation}, the sum in the definition of $\tilde \Phi$ is finite if $T\in \cstu[X_\gamma]$, and $\tilde \Phi$ is well-defined. Writing $\beta$ for $\gamma^k$, to prove that $\tilde \Phi$ is self-adjoint, note that with a change of variable one has 
		\[
		E(v_\beta^* T)=\sum_i (T\delta_i|v_\beta \delta_i)p_i=\sum_i (T \delta_{\beta^{-1}(i)}|\delta_i) p_{\beta^{-1} (i)}.
		\] 
		Since $\dot E$ is self-adjoint and $p_{\beta^{-1}(i)} v_\beta^*= \delta_{\beta^{-1}(i)}\odot \delta_i$ we have  
		\[
		(v_\beta E( T^* v_\beta))^*= E( v_\beta^* T) v_\beta^* = 
		\sum_i (T \delta_{\beta^{-1}(i)}|\delta_i) \delta_{\beta^{-1}(i)}\odot \delta_i
		\] 
With another change of variable in \eqref{eq.tildePhi} ($k\mapsto -k$) we obtain $\tilde\Phi(T^*)^*=\tilde \Phi(T)$. 

To prove multiplicativity, fix $S$ and $T$ in $\cstu[X_\gamma]$ and $k,l$ in $\bbZ$. Note that 
\[
v_\gamma^k E(S v_\gamma^k) v_\gamma^l E(T v_\gamma^l)
=\left(\sum_i (S\delta_{\gamma^k(i)}|\delta_i) \delta_{\gamma^k(i)}\odot \delta_i\right)
\left(\sum_j (T\delta_{\gamma^l(j)} |\delta_j) \delta_{\gamma^l(j)}\odot \delta_j\right).
		\]
		We have that $( \delta_{\gamma^k(i)}\odot \delta_i)(\delta_{\gamma^l(j)}\odot \delta_j)\neq 0$ if and only if $i=\gamma^l(j)$, and therefore the right-hand side of the last displayed formula is equal to 
		\[
		\sum_k \sum_l \sum_j (S\delta_{\gamma^{k+l}(j)}|\delta_{\gamma^l(j)} )
		(T\delta^l(j)|\delta_j)\delta_{\gamma^{k+l}(j)}\odot \delta_j. 
		\]
		Finally, $(ST\xi|\eta)=\sum_i (T\xi|\delta_i)(\delta_i|S\eta) $. By putting this together one obtains that $\tilde\Phi$ is multiplicative.

		In the next step we prove that $\|\tilde\Phi(\dot T)\|=\|\dot T\|$ for all $\dot T$ in a norm-dense subspace of $\roeq[X_\gamma]$. 
		For $A\subseteq \bbN$ let 
		\[
		p_A=\proj_{\overline{\Span\{\delta_i\mid i\in A\}}},
		\]
		the projection to the closure of $\Span\{\delta_i\mid i\in A\}$.\footnote{Thus $p_i$ as used earlier is really $p_{\{i\}}$; this will hopefully not lead to confusion.}
		If $m\in \bbN$, $\cA$ is a partition of $\bbN$ into finitely many pieces, and $\lambda_{k,A}$ are scalars for $|k|\leq m$ and $A\in \cA$, then the operator (writing $v_\gamma^{-k}=(v_\gamma^*)^k$) 
		\[
		T=\sum_{|k|\leq m} v_\gamma^k \sum_{A\in \cA} \lambda_{k,A} p_A
		\]
		belongs to $\cstu[X_\gamma]$. Operators of this form form a norm-dense *-subalgebra of $\cstu[X_\gamma]$. We will prove $\|\tilde\Phi(\dot T)\|=\|\dot T\|$ for such $T$. 
		First, 
		\[
		\tilde\Phi(\dot T)=\sum_{|k|\leq m} \dot v_\gamma^k \sum_{A\in \cA} \lambda_{k,A} p_{\Phi(\dot A)}.
		\] 
		For $T\in \cB(H)$, we have that 
		\begin{equation}
		\label{eq.dotT}
		\|\dot T\|=\lim_{n\to \infty}\sup\{ \|T\xi\|\mid \|\xi\|_2=1, \min\supp(\xi)\geq n\}. 
		\end{equation}
		Let $\Phi_*$ be a (set-theoretic) lifting of $\Phi$, so that $[\Phi_*(A)]_{\Fin}=\Phi([A]_{\Fin})$ for all $A$.  
		
		Then with 
		\[
		T_1=\sum_{|k|\leq m} v_\gamma^k \sum_{A\in \cA} \lambda_{k,A} p_{\Phi_*(A)}
		\]
		we have $\tilde \Phi(\dot T)=\dot T_1$. We need to prove that $\|\dot T\|=\|\dot T_1\|$. 
		
		Let $\cA'=\{\Phi_*(A)\mid A\in \cA\}$. 
		Let $\cB$ be the Boolean algebra generated by $\gamma^k [\cA]$, for $|k|\leq m$ and let $\cB'$ be the Boolean algebra generated by $(\gamma')^k [\cA']$, for $|k|\leq m$. 
		
		Since $\Phi$ is an automorphism of $\pnmf$, for every large enough $n$ every $B\in \cB$ satisfies $B\setminus n\neq \emptyset$ if and only if $B$ is infinite if and only if $\Phi_*(B)\setminus n\neq \emptyset$ . There is then an automorphism of $\cP(\bbN\setminus n)$ that sends $\gamma^k(A\setminus n)$ to $(\gamma')^k(\Phi_*(A)\setminus m)$ for all $A\in \cB$ and all $k\leq m$. 
		This automorphism is implemented by a permutation of $\bbN\setminus n$. This permutation sends every $\xi\in \ell_2$ with $\min(\supp(\xi))\geq n$ to a vector $\xi'$ with $\min(\supp(\xi'))\geq n$, $\|\xi\|_2=\|\xi'\|_2$, and $\|T(\xi)\|_2=\|T_1(\xi')\|_2$. 
			 Since this holds for arbitrarily large $n$, by \eqref{eq.dotT} we have that $\|\dot T\|=\|\dot T_1\|$ as required. 
		 
		 We have now proven that $\tilde \Phi$ is a *-homomorphism between dense *-subalgebras of $\roeq[X_\gamma]$ and $\roeq[X_{\gamma'}]$. It clearly  sends $v_\gamma$ to $v_{\gamma'}$. This completes the proof. 
		\end{proof}

By Claim, $\tilde\Phi$  extends to an isomorphism between $\roeq(X_\gamma)$ and $\roeq(X_{\gamma'})$.  Clearly $\tilde \Phi(\dot v_\gamma)=\dot v_{\gamma'}$. 
\end{proof}

The following is one of equivalent characterizations of the asymptotic dimension, as given in \cite[Theorem 9.9]{RoeBook}.  
A metric space $X$ has \emph{asymptotic dimension $\leq d$} if  for every $k$ it has an open covering $\cU$ such that the supremum of diameters of elements of $\cU$ is finite and each  $k$-ball intersects at most $d+1$ elements of $\cU$. 
\emph{Property~A} is an amenability notion for coarse spaces. It is a  consequence of having finite asymptotic dimension (see e.g., \cite[Definition 11.35]{RoeBook}, \cite[Remark 11.36 (ii)]{RoeBook}) and it has strong rigidity consequences.  For example, the first one in a line of rigidity results for uniform Roe algebras that culminated in \cite{baudier2022uniform} and \cite{martinez2025Cstar} was obtained  for spaces with property A in \cite{spakula2013rigidity}.  

\begin{theorem}\label{T.uRQ+}
There are $2^{\aleph_0}$ coarsely inequivalent uniformly locally finite spaces of asymptotic dimension $1$ (and therefore with Property A)  such that \ch{} implies all of their uniform Roe coronas are isomorphic, while $\OCAT$ and $\MA$ imply that all of their uniform Roe coronas are nonisomorphic.\footnote{We will not define $\OCAT$ and $\MA$ here, since their use is reduced to quoting a result from~\cite{braga2018uniformRoe} and \cite{baudier2022uniform}. Suffice it to say that if $\ZFC$ has a model then so does $\ZFC+\OCAT+\MA$. Our contribution is in constructing the spaces and proving that \ch{} implies their uniform Roe coronas are isomorphic.}
\end{theorem}

\begin{proof} By Corollary~\ref{C.many.nonconjugate}  there are increasing sequences $\bar n(r)$, for $r\in \cP(\bbN)$, such that for all further subsequences $\bar n’(r)$ of $\bar n(r)$,  \ch{} implies that the dynamical systems $\fA_{\bar n’(r)}$, for $r\in \cP(\bbN)$, are isomorphic but they are not isomorphic if all automorphisms of $\pnmf$ are trivial. 
By passing to subsequences we may assume $n_{i+1}(r)>i\cdot  n_i(r)$ for all $i$ and~$r$. 
 We claim that the coarse spaces $X(r)=X(\bar n(r))$, for $r\in \cP(\bbN)$,  are as required.

Assume \ch. By the choice of $\bar n(r)$, all $\fA_{\bar n(r)}$ are isomorphic.     Therefore Lemma~\ref{L.tildePhi} implies that all $\roeq(X(r))$ are isomorphic. 

	On the other hand,  \cite[Theorem~1.5]{baudier2022uniform} implies that  under $\OCAT$ and $\MA$ we have that $\roeq(X)\cong \roeq(Y)$ implies that $X$ and $Y$ are coarsely equivalent.  Lemma~\ref{L.coarse.rotary} implies that the spaces $X(r)$ are pairwise coarsely inequivalent. 

It remains to verify that these spaces have the required coarse properties. 
The spaces $X(r)$ are clearly locally uniformly finite.
Fix $r$ and let $I_j$, for $j\in \bbN$,  be intervals of $\bbN$ of length $n_j$ that are coarse components of $X(r)$. In order to see that the asymptotic dimension of $X(r)$ is 1, fix $k\geq 1$.  Partition each $I_j$ into intervals of length between $k$ and $2k$. These intervals form an open cover of $X(r)$ of sets whose diameters are $\leq 2k$, and each $k$-ball intersects at most two of them.  By \cite[Defiition~9.4]{RoeBook} each  $X(r)$ has asymptotic dimension $\leq 1$ and  by \cite[Remark 11.36 (ii)]{RoeBook} each $X(r)$ has property A.  The argument that the asymptotic dimension of~$X(r)$ is nonzero is easy and therefore omitted.  
	\end{proof}

The main result of \cite{brian2024does}, together with Lemma~\ref{L.tildePhi},  shows that $\roeq(\bbN)$ (where $\bbN$ is taken with the standard metric) has an automorphism that sends the left unilateral shift to the right unilateral shift (i.e., it reverses the $K_1$-group of the algebra). It is a major open problem, dating back to \cite{BrDoFi}, whether the Calkin algebra has such an automorphism. Regrettably, the only automorphisms of $\pnmf$, identified with the lattice of projections of $\ell_\infty/c_0$, that extend to the Calkin algebra are the trivial ones (see \cite[Notes to \S 17.9]{Fa:STCstar} for more details). 

\section{Concluding remarks}

A positive answer to the following would imply that the requirement on parities in  Theorem~\ref{thm:ExtensionToTriv} is redundant. 

\begin{question}\label{Q.parity}
	Are there trivial automorphisms $\alpha_f$ and $\alpha_g$ such that $\fA_f \equiv \fA_g$ but $\parity(\alpha_f) \neq \parity(\alpha_g)$? 
\end{question}

Note that if there are such automorphisms then the maps $f = s_{\Z \times \Z}$ and $g = s_{\Z \times \Z} \oplus s$ provide an  example. To prove this, assume that $\fA_f\equiv\fA_g$ but $\parity(\alpha_f)$ is even while $\parity(\alpha_g)$ is odd. Since  $Z(\alpha)$ is definable as the supremum of atoms of $\FIX_\alpha$ (although the fact that this supremum exists is an accident), the  structures $(\cP(Z(\alpha))/\Fin,\alpha\rs Z(\alpha))$  and $(\cP(Z(\beta))/\Fin,\alpha\rs Z(\beta))$ are elementarily equivalent but with different parities. Being elementarily equivalent, they have the same number of atoms. Since the parities are different, this number is infinite and each of $\fA_\alpha$ and $\fA_\beta$ is a direct sum of $s_{\Z \times \Z}$ and a finite number of copies of $\fA_\sigma$ and $\fA_{\sigma^{-1}}$. Finally, since $\fA_\sigma$ and $\fA_{\sigma^{-1}}$ are elementarily equivalent, and since by the Feferman--Vaught theorem (\cite{feferman1959first}, also \cite[Theorem~6.3.4]{ChaKe}) the first-order theory is unchanged if a direct summand is replaced by one that is elementarily equivalent,~$\fA_\alpha$ is elementarily equivalent to $s_{\Z \times \Z}$ while $\fA_\beta$ is elementarily equivalent to $s_{\Z \times \Z}\oplus \s$.

We conjecture that the answer to Question~\ref{Q.parity} is yes.

In  \S\ref{sec:RotaryTheory}  we presented some partial results towards describing first-order properties of~$\fA_\alpha$. Since in \S \ref{S.Saturation}  we have shown that under \ch potential isomorphism between dynamical systems associated with trivial automorphisms of $\cPNF$ reduces to elementary equivalence,  it would be desirable to find a complete characterization of such properties. The paragraph preceding Problem~\ref{Prob.Characterization}  below suggests that such complete characterization exists. 

Let us say that there is an \emph{arithmetic obstruction} to the conjugacy of two rotary maps $\alpha_{\bar m}$ and $\alpha_{\bar n}$ if there exist some $j < k$ in $\w$ such that one of the formulas described in Theorem~\ref{thm:RotaryObstruction} has a different truth value for $\fA_{\bar m}$ and $\fA_{\bar n}$. 

\begin{question} Is the theory of a rotary map completely determined by the arithmetic obstructions that follow from it? 
\end{question}

A positive answer would imply that the absence of arithmetic obstructions to the conjugacy of two rotary maps $\alpha_{\bar m}$ and $\alpha_{\bar n}$ implies $\fA_{\bar m}$ and $\fA_{\bar n}$ are elementarily equivalent, and therefore conjugate under \ch. 
Observe that there are many pairs of sequences $\bar m$ and $\bar n$ such that there is no arithmetic obstruction to their conjugacy, but they are not trivially conjugate (This is of course the idea behind the proof of  the second part of Corollary~\ref{C.many.nonconjugate}.) In such cases we do not know whether \ch implies $\fA_{\bar m} \cong \fA_{\bar n}$. Two concrete examples are recorded in the following question, phrased without mentioning \ch. 

\begin{question} \label{Q.examples}
	\begin{enumerate}
		\item Are $\fA_{(2i)!}$ and $\fA_{(2i+1)!}$ elementarily equivalent? 
		
		\item Suppose that $p_i$ and $q_i$ are increasing sequences of primes such that for every $d\geq 2$ some $k(d)$ satisfies $p_i=q_i (\mathrm{mod}\  d)$ for all $i\geq k(d)$. Does it follow that $\fA_{(p_i)}$ and $\fA_{(q_i)}$ are elementarily equivalent? 
	\end{enumerate}
\end{question}

Note that a more na\" ive version of the second part of the question has negative answer, because Dirichlet’s theorem implies that all $d\geq 2$ and $0\leq l<d$ there are infinitely many primes $p$ such that $p= l (\mathrm{mod}\ d)$.

We conclude by outlining a more ambitious variant of Question~\ref{Q.examples}. 
By the Feferman--Vaught theorem (\cite{feferman1959first},  \cite[Proposition~6.3.2]{ChaKe}), the theory of a reduced product $\prod_n A_n/\Fin$ is computable from the sequence of theories of~$A_n$ (or even, if these theories converge, from their limit; this is a key point in Ghasemi’s Trick,  Proposition~\ref{P.FVtrick}).  A more precise consequence of the Feferman--Vaught theorem (and the fact that the theory of atomless Boolean algebras admits elimination of quantifiers) is that $\prod_i M_i/\Fin$ and $\prod_i N_i/\Fin$ are elementarily equivalent if and only if the set of limit points of the theories of the $M_i$'s  is equal to the set of limit points of the theories of the $N_i$'s (equivalently, for every $\calL$-sentence $\varphi$, the set $\{i\mid M_i\models \varphi\}$ is infinite if and only if the set $\{i\mid N_i\models \varphi\}$ is infinite).  Since~$\fA_{\bar n}$ is the reduced product of finite structures (Lemma~\ref{L.ReducedProduct}), this observation shows that the theory of a rotary automorphism is decidable and reduces the problem of computing it  to a problem in finite model theory (see e.g.,~\cite{gradel2007finite}).   The \emph{monadic second-order logic} is the extension of the first-order logic in which quantification over subsets of the domain is allowed.  
In particular,  quantifying over the elements of $\cP(n_i)$ is interchangeable with quantifying over the subsets of $n_i$. As in Lemma~\ref{L.ReducedProduct}, $r_{n_j}$ denotes the automorphism of $\cP(n_j)$ obtained by cycling its atoms. 

\begin{problem}  \label{Prob.Characterization} Describe the sequences $(n_j)$ such that the theories of $ \< \P(n_j),r_{n_j} \>$ converge. Equivalently, describe the sequences $(n_i)$ such that the monadic second-order theory of the (directed) $n_i$-cycles converge. 
\end{problem}

\bibliographystyle{amsplain}
\bibliography{10-conjugacy}

\providecommand{\bysame}{\leavevmode\hbox to3em{\hrulefill}\thinspace}
\providecommand{\MR}{\relax\ifhmode\unskip\space\fi MR }
\providecommand{\MRhref}[2]{%
  \href{http://www.ams.org/mathscinet-getitem?mr=#1}{#2}
}
\providecommand{\href}[2]{#2}
\begin{thebibliography}{10}

\bibitem{Anderson}
R.~D. Anderson, \emph{The algebraic simplicity of certain groups of
  homeomorphisms}, American Journal of Mathematics \textbf{80} (1958),
  955--963.

\bibitem{baudier2022uniform}
F.P. Baudier, B.M. Braga, I.~Farah, A.~Khukhro, A.~Vignati, and R.~Willett,
  \emph{Uniform {R}oe algebras of uniformly locally finite metric spaces are
  rigid}, Inv. Math. \textbf{230} (2022), no.~3, 1071--1100.

\bibitem{Darji}
N.~C. Bernardes~Jr. and U.~B. Darji, \emph{Graph theoretic structure of maps of
  the cantor space}, Adv. Math. \textbf{231} (2012), 1655--1680.

\bibitem{braga2018uniformRoe}
B.~M. Braga, I.~Farah, and A.~Vignati, \emph{Uniform {R}oe coronas}, Adv. Math.
  \textbf{389} (2021), Paper No. 107886, 35.

\bibitem{brian2024does}
W.~Brian, \emph{Can ${\cP}(\omega)/{\Fin}$ tell its right hand from its left?},
  arXiv, 2024.

\bibitem{brianPsets}
W.~R. Brian, \emph{{$P$}-sets and minimal right ideals in {$\mathbb N^*$}},
  Fund. Math. \textbf{229} (2015), 277--293.

\bibitem{BrDoFi}
L.G. Brown, R.G. Douglas, and P.A. Fillmore, \emph{Extensions of
  \cstar-algebras and {K}-homology}, Annals of Math. \textbf{105} (1977),
  265--324.

\bibitem{ChaKe}
C.~C. Chang and H.~J. Keisler, \emph{Model theory}, third ed., Studies in Logic
  and the Foundations of Mathematics, vol.~73, North-Holland Publishing Co.,
  Amsterdam, 1990.

\bibitem{cornulier2019near}
Y.~de Cornulier, \emph{Near actions}, preprint available at
  \texttt{https://arxiv.org/abs/1901.05065} (2019).

\bibitem{de2023trivial}
B.~de~Bondt, I.~Farah, and A.~Vignati, \emph{Trivial automorphisms of reduced
  products}, Israel J. Math. (to appear).

\bibitem{Fa:STCstar}
I.~Farah, \emph{Combinatorial set theory and \cstar-algebras}, Springer
  Monographs in Mathematics, Springer, 2019.

\bibitem{farah2022corona}
I.~Farah, S.~Ghasemi, A.~Vaccaro, and A.~Vignati, \emph{Corona rigidity}, Bull.
  Symbl. Logic. (to appear).

\bibitem{feferman1959first}
S.~Feferman and R.~Vaught, \emph{The first order properties of products of
  algebraic systems}, Fund. Math. \textbf{47} (1959), no.~1, 57--103.

\bibitem{Fuchino}
S.~Fuchino, \emph{Doctoral dissertation}, Freie Universit\"{a}t Berlin (1985).

\bibitem{Gha:FDD}
S.~Ghasemi, \emph{Isomorphisms of quotients of {FDD}-algebras}, Israel J. Math.
  \textbf{209} (2015), no.~2, 825--854.

\bibitem{ghasemi2016reduced}
\bysame, \emph{Reduced products of metric structures: a metric
  {F}eferman--{V}aught theorem}, J. Symb. Log. \textbf{81} (2016), no.~3,
  856--875.

\bibitem{GoodMeddaugh}
C.~Good and J.~Meddaugh, \emph{Shifts of finite type as fundamental objects in
  the theory of shadowing}, Invent. Math. \textbf{220} (2020), 715--736.

\bibitem{gradel2007finite}
E.~Gr{\"a}del, P.G. Kolaitis, L.~Libkin, M.~Marx, J.~Spencer, M.Y. Vardi,
  Y.~Venema, and S.~Weinstein, \emph{Finite model theory and its applications},
  Texts in Theoretical Computer Science. An EATCS Series, Springer, 2007.

\bibitem{JonOl:Almost}
B.~Jonsson and P.~Olin, \emph{Almost direct products and saturation},
  Compositio Math. \textbf{20} (1968), 125--132.

\bibitem{Ku:Book}
K.~Kunen, \emph{Set theory: An introduction to independence proofs},
  North--Holland, 1980.

\bibitem{martinez2024rigidity}
D.~Mart\'inez and F.~Vigolo, \emph{A rigidity framework for {R}oe-like
  algebras}, arXiv preprint arXiv:2403.13624 (2024).

\bibitem{martinez2025Cstar}
\bysame, \emph{\cstar-rigidity of bounded geometry metric spaces}, Publ. Math.
  Inst. Hautes \'Etudes Sci. \textbf{141} (2025), 333--348.

\bibitem{mckenney2018forcing}
P.~McKenney and A.~Vignati, \emph{Forcing axioms and coronas of
  {$\mathrm{C}^*$}-algebras}, J. Math. Log. \textbf{21} (2021), no.~2, Paper
  No. 2150006, 73. \MR{4290495}

\bibitem{vanMill}
J.~van Mill, \emph{Handbook of set-theoretic topology}, ch.~An introduction to
  $\beta \omega$, pp.~503--560, North-Holland, 1984, eds. K. Kunen and J. E.
  Vaughan.

\bibitem{palmgren}
E.~Palmgren, \emph{A direct proof that certain reduced products are countably
  saturated}, Tech. report, {Uppsala University. Department of Mathematics}
  Report, 1994.

\bibitem{Pede:Analysis}
G.K. Pedersen, \emph{Analysis now}, Graduate Texts in Mathematics, vol. 118,
  Springer-Verlag, New York, 1989.

\bibitem{RoeBook}
J.~Roe, \emph{Lectures on coarse geometry}, University Lecture Series, vol.~31,
  American Mathematical Society, Providence, RI, 2003.

\bibitem{RubinStepanek}
M.~Rubin and P.~\v{S}t\v{e}p\'anek, \emph{Handbook of boolean algebras},
  ch.~Homogeneous Boolean algebras, pp.~679--715, North-Holland, Amsterdam,
  1989, ed. J. D. Monk with R. Bonnet.

\bibitem{ShelahAutomorphisms}
S.~Shelah, \emph{Proper forcing}, Lecture Notes in Mathematics, vol. 940,
  Springer-Verlag, Berlin, 1982.

\bibitem{ShelahSteprans}
S.~Shelah and J.~Stepr\={a}ns, \emph{$\mathsf{PFA}$ implies all automorphisms
  are trivial}, Proc. Amer. Math. Soc. \textbf{104} (1988), 1220--1225.

\bibitem{spakula2013rigidity}
J.~{\v{S}}pakula and R.~Willett, \emph{On rigidity of {R}oe algebras}, Adv.
  Math. \textbf{249} (2013), 289--310.

\bibitem{van1990automorphism}
E.K. van Douwen, \emph{The automorphism group of {$P (\omega)/\Fin$} need not
  be simple}, Top. Appl. \textbf{34} (1990), no.~1, 97--103.

\bibitem{Ve:OCA}
B.~Veli{\v{c}}kovi{\'c}, \emph{{OCA} and automorphisms of {${\mathcal
  P}(\omega) /\Fin$}}, Top. Appl. \textbf{49} (1992), 1--13.

\bibitem{willett2020higher}
R.~Willett and G.~Yu, \emph{Higher index theory}, vol. 189, Cambridge
  University Press, 2020.

\end{thebibliography}
\end{document}